\documentclass[twocolumn]{autart}    

\usepackage{graphicx}          

\usepackage{amsmath,amsfonts,amssymb}
\usepackage{subfig}
\usepackage{multirow}
\usepackage[round]{natbib}
\usepackage[bookmarks=false]{hyperref}

\newtheorem{assumption}{Assumption}
\newtheorem{remark}{Remark}
\newtheorem{definition}{Definition}
\newtheorem{theorem}{Theorem}
\newtheorem{lemma}{Lemma}
\newtheorem{proposition}{Proposition}
\newtheorem{corollary}{Corollary}
\newtheorem{problem}{Problem}

\newenvironment{proof}{\textit{Proof:} }{\hfill$\square$}

\newcommand{\R}{\mathbb{R}}
\DeclareMathOperator{\sign}{sign}

\begin{document}

\begin{frontmatter}

\title{Learning Maximal Regions of Attraction on Compact Sets using Taylor Lyapunov Networks\thanksref{footnoteinfo}} 

\thanks[footnoteinfo]{This paper was not presented at any IFAC 
meeting. Corresponding author M.~Barreau. This work was partially supported by the Wallenberg AI, Autonomous Systems and Software Program (WASP) funded by the Knut and Alice Wallenberg Foundation, and by the European Union’s Horizon Research and Innovation Actions program under grant agreement No. 101070162.}

\author[KTH]{Matthieu Barreau}\ead{barreau@kth.se},    
\author[UCL]{Nicola Bastianello}\ead{n.bastianello@ucl.ac.uk}               

\address[KTH]{Digital Futures and KTH Royal Institute of Technology, Department of Decision and Control Systems, Stockholm, Sweden}  
\address[UCL]{Department of Electronic \& Electrical Engineering, University College London, United Kingdom}


\begin{abstract}                          
We consider the problem of learning maximal Lyapunov functions to compute the maximal region of attraction of nonlinear dynamical systems using neural networks. Two obstacles limit existing neural-network approaches: 1) finite-width approximations can fail to preserve the strict Lyapunov inequalities near the equilibrium, where they become arbitrarily weak, and 2) training is necessarily performed on a bounded computational domain for which the classical Lyapunov and Zubov theorems do not directly apply. To address these issues, first we derive Lyapunov and Zubov theorems directly on a prescribed compact computational domain, characterizing the maximal region of attraction within that domain as the solution of a boundary-value PDE. Second, we introduce the Taylor Lyapunov network, an architecture that prescribes the quadratic leading-order term via the system's linearization and represents the higher-order remainder through a structured cubic factorization, guaranteeing that any network in this class is a valid local Lyapunov function by construction. Finally, we prove that Taylor Lyapunov networks are universal approximators of maximal Lyapunov functions on compact domains, and combine these results into an unsupervised, physics-informed training procedure. Numerical experiments on systems of varying dimensions and stability types show that the proposed architecture is efficient at estimating the maximum region of attraction.
\end{abstract}

\end{frontmatter}

\section{Introduction}

For a nonlinear dynamical system, determining the region of attraction of an equilibrium is important for safe behavior but generally difficult to characterize. Lyapunov functions provide a principled way to obtain certified inner approximations of this region, but their construction becomes increasingly challenging for nonlinear and high-dimensional systems \citep{khalil2002nonlinear}. A solution is to consider surrogates such as rational functions \citep{vannelli1985maximal} or sum-of-squares \citep{henrion2013convex}. Polynomial and rational parameterizations provide powerful computational tools, but their expressiveness and computational tractability can become limiting for complex \citep{ahmadi2011globally} or high-dimensional systems \citep{ahmadi2019dsos}. This motivates considering more flexible function classes, such as neural networks.

Learning neural Lyapunov functions has been investigated \citep[see, e.g.,][]{325708,LarsGrune,richards2018lyapunov,gaby2022lyapunovnet}. Rather than seeking an arbitrary Lyapunov function, following \citet{vannelli1985maximal} and \citet{zubov1961methods}, we refer to a maximal Lyapunov function whose relevant sublevel set coincides with the maximal region of attraction.
In \citet{liu2023physics}, the authors propose learning maximal neural Lyapunov functions leveraging the physics-informed machine learning paradigm \citep{karniadakis2021physics} to incorporate physical priors into the training procedure as soft dynamical constraints.
Physics-informed neural networks offer universal approximation capabilities on compact sets under appropriate regularity assumptions \citep{hornik1991approximation,mishra2023estimates}, making them particularly useful in many contexts, including for control \citep{pimlcontrol}. However, two theoretical issues remain.

First, a standard multi-layer perceptron can approximate any Lyapunov function arbitrarily well on a compact domain given enough neurons \citep{hornik1991approximation}. However, this does not guarantee that a finite-width approximation preserves the strict inequalities required for a Lyapunov function, especially near the equilibrium where they become arbitrarily weak. Consequently, a small approximation error can jeopardize positive definiteness or the negativity of the orbital derivative in a neighborhood of the equilibrium. 

Second, neural approximation is performed on a prescribed compact computational domain, whereas classical Lyapunov and Zubov characterizations do not hold in such bounded domains. If a trajectory leaves the computational domain, there is no guarantee of the Lyapunov function decay. We therefore require a domain-aware formulation of Lyapunov and Zubov theory.

\subsection{Related works}

We refer to \citet{pimlcontrol} and \citet{dawson2023safe} for comprehensive surveys and
highlight the most relevant works below.

\textit{Neural Lyapunov architectures.}
Neural Lyapunov functions were first learned in \citet{LarsGrune} for globally
stable systems, with soft constraints on the Lyapunov conditions and
approximation guarantees. Several works then enforce properties by design.
In \citet{kolter2019learning}, the architecture imposes hard constraints, but it
targets the accuracy of learned stable dynamics rather than RoA estimation and
applies only to a subclass of globally stable systems
\citep{jongeneel2024continuation}.
Positive definiteness is enforced in \citet{gaby2022lyapunovnet} and \citet{akhtar2026learning},
and more generally through PLnets \citep{wang2024monotone} in
\citet{cheng2024learning}. However, none of these enforces the condition on the
orbital derivative, which is the hardest one, especially at the equilibrium where
the inequality collapses. Other architectures target the training algorithm:
extreme learning machines \citep{zhou2024physics} for efficient training, and
projection layers \citep{min2024hardnet,goertzen2026hardnet} to embed hard
constraints in the optimizer. The latter do not extend straightforwardly to
dynamical constraints such as negative definiteness of the orbital derivative.

\textit{Verification.}
When constraints are only enforced softly, validity of the learned function is not
guaranteed. Counterexample-guided synthesis with a symbolic verifier
\citep{abate2020formal} addresses this, but only for globally asymptotically
stable systems of limited dimension. SMT-based verification
\citep{zhou2022neural,liu2025formally} provides RoA estimates with a posteriori
certificates, but suffers from the curse of dimensionality
\citep{ahmed2020automated}, which undermines the main motivation for neural
Lyapunov functions.

\textit{Maximal Lyapunov functions and compact domains.}
Zubov's theorem \citep{zubov1961methods} was leveraged in \citet{liu2023physics}
to learn Lyapunov functions that maximize the RoA. The compact computational
domain and the collapse of the Lyapunov inequality at the origin were
acknowledged in \citet{gaby2022lyapunovnet} through modified definitions.
However, training is performed on a compact set while the theorems used are
stated on $\mathbb{R}^n$, so the learned function may fail to be a Lyapunov
function over the predicted RoA. An a priori compact domain is available only
in special cases, such as linear or homogeneous systems
\citep{lima2024physics,zenkin2025neural} and some systems on manifolds
\citep{akhtar2026learning}. In general, the fundamental theorems must be
rewritten for a prescribed compact set.

\subsection{Contributions}

To the best of the authors' knowledge, no existing work combines a domain-aware Lyapunov/Zubov theory with an architecture that guarantees local Lyapunov properties while retaining universal approximation. The paper makes three contributions. First, we introduce a Taylor neural architecture in which the leading quadratic term is prescribed by the linearization, and the higher-order residual is factored into cubic monomials. This enforces the correct local order of the approximation and yields local Lyapunov validity for sufficiently accurate finite-width approximations. Second, we formulate Lyapunov and Zubov characterizations on a prescribed compact domain. Third, we combine these ingredients into a physics-informed training procedure to learn maximal Lyapunov functions and evaluate the resulting method on systems with varying stability and regions of attraction.

\subsection{Organization and notation}

Section~\ref{sec:preliminaries} reviews the preliminaries and Section~\ref{sec:problem_formulation} formulates the problem. Section~\ref{sec:zubov_constrained} formulates the Zubov theorem on compact sets. Section~\ref{sec:neural-lyapunov} proposes the Taylor Lyapunov architecture and its universal approximation properties. Section~\ref{sec:practical_implementation} introduces a practical implementation, Section~\ref{sec:numerical} presents numerical simulations, and Section~\ref{sec:conclusions} concludes the paper.

Throughout the paper, $\R$ denotes the set of real numbers, and $C^k(I, J)$ the set of $k$-times differentiable functions from $I$ to $J$. For $x = \left[ x_1 \ \cdots \ \ x_n \right]^{\top} \in \R^n$ and $\psi: \R \to \R$, we use the notation $\psi.(x) = \left[ \psi(x_1) \ \cdots \ \psi(x_n) \right]^{\top}$, referring to the element-wise operation.
For $x, y \in \R^n$, the Euclidean scalar product is written as $x \cdot y$, and its Euclidean norm is defined as $\| x \| = \sqrt{ x \cdot x}$. For two squared symmetric matrices $A$ and $B$ of the same size, $A \prec B$ means that $A-B$ has strictly negative eigenvalues. We denote the boundary of a set $\mathcal{D} \subset \R^n$ by $\partial \mathcal{D}$ and the interior of $\mathcal{D}$ is written as $\mathcal{D}^{\circ}$.

\section{Preliminaries}
\label{sec:preliminaries}

We consider the following dynamical system:
\begin{equation} \label{eq:dynamical_system}
    \left\{ \begin{array}{ll}
         \dot{x}(t) = f(x(t)),  \quad & t \geq 0, \\
         x(0) = x_0 \in \R^n, & 
    \end{array} \right.
\end{equation}
where $n \in \mathbb{N} \backslash \{0\}$ is the dimension of the system, and $f: \mathbb{R}^n \to \mathbb{R}^n$ is Lipschitz continuous and, in general, non-linear. Under these conditions, a unique class $C^1$ solution exists to~\eqref{eq:dynamical_system} which is forward complete \citep{angeli1999forward}. We then denote by $t \mapsto \pi(x_0, t)$ the unique trajectory of \eqref{eq:dynamical_system} which starts at $x_0$. Additionally, we assume without loss of generality that the origin is an equilibrium point of $f$, i.e., $f(0) = 0$. We make the following assumption on the stability of the origin.

\begin{assumption} \label{ass:linearization}
    $f$ in~\eqref{eq:dynamical_system} can be written as $f(x) = Ax + o( \| x \|)$ such that $A$ has all eigenvalues with strictly negative real part.
\end{assumption}

With this assumption, it is well known that the origin is an exponentially stable equilibrium point \citep{khalil2002nonlinear}. Therefore, we can define its maximal region of attraction (RoA) as the region of the state space from which any trajectory converges to the origin, i.e., $\mathcal{R}^* = \{ x_0 \in \R^n \mid \lim_{t \to \infty} \| \pi(x_0, t) \| = 0 \}$.
If $\mathcal{R}^*$ is not the full state-space, finding an analytical expression of the RoA is difficult for most systems. Thus, we commonly rely on Lyapunov functions to characterize it. The conventional definition (as provided in \citet[Theorem 4.1]{khalil2002nonlinear}, for instance) is given below.

\begin{definition}\label{def:conventional-lyapunov-function}
    A continuously differentiable function $V: \R^n \to \mathbb{R}^+$ is said to be a \textbf{Lyapunov function for $f$} if there exists a $d > 0$ such that
    \begin{subequations} \label{eq:lyapunov_definition}
        \begin{flalign}
            &V(0) = 0, \label{eq:definite} \\
            &V(x) > 0 \text{ for all } x \in \R^n \setminus \{0\} \text{ s.t. } V(x) \leq d, \label{eq:positive} \\
            &\nabla V(x) \cdot f(x) < 0 \text{ for all } x \in \R^n \setminus\{0\} \text{ s.t. } V(x) < d. \label{eq:decreasing}
        \end{flalign}
    \end{subequations}
\end{definition}
By Theorem 4.1 from \citet{khalil2002nonlinear}, we get that the set $\mathcal{R}(V, d) = \{ x \in \R^n \mid V(x) < d \} \subseteq \mathcal{R}^*$ is a RoA around the origin for $\mathcal{D} = \R^n$. Thanks to Assumption~\ref{ass:linearization}, there exists a quadratic Lyapunov function $V_q(x) = x^{\top} P x$ with $P \succ 0$ \citep[Theorem 12.6]{glad2018control}. However, any Lyapunov function does not characterize the maximal region of attraction. One has to rely on maximal Lyapunov functions (or Zubov functions) for which the $1$-level set is exactly the maximal RoA. The following theorem from \citet{zubov1961methods} and \citet{vannelli1985maximal} defines this function.

\begin{theorem}
    \label{thm:zubov}
    Let $f$ in \eqref{eq:dynamical_system} be Lipschitz continuous. An open set $A$ containing the origin is the maximal RoA around the origin if and only if there exists a continuous function $V_m$ (maximal Lyapunov function) and a positive definite function $\phi$ such that:
    \begin{enumerate}
        \item[i)] $V_m(0) = 0$, $V_m(x) > 0$ for all $x \in A \setminus \{0\}$,
        \item[ii)] $V_m(x) \to 1$ as $x \to \partial A$ and/or $\|x\| \to \infty$,
        \item[iii)] $\dot{V}_m(x)$ is well defined at all $x \in A$ and satisfies:
        \begin{equation}
            \label{eq:negative_zubov}
            \dot{V}_m(x) = - \phi(x) (1 - W(x)).
        \end{equation}
    \end{enumerate}
\end{theorem}

\section{Problem statement}
\label{sec:problem_formulation}

This paper focuses on adapting Theorem~\ref{thm:zubov} for computational tractability. Before formulating the problem, we propose additional preliminary definitions.

\subsection{A numerical artifact: the computational domain}

One practical difficulty when constructing a surrogate for a maximal Lyapunov function is to consider a compact set on which we will perform the computations. We denote these sets computational domains and, up to rescaling, they are defined as $\mathcal{D} = \{ x \in \R^n \mid \| x \|_p^p \leq 1 \}$ where $x = [ x_1 \ \cdots \ x_n]^{\top}$ with the $p$-norm $\| x \|_p = ( \sum_{i=1}^n x_i^p )^{1/p}$.

\begin{remark}
    \label{rem:continuous}
    Note that since $1 < p < \infty$, $\| \cdot \|_p^p$ is smooth on $\partial \mathcal{D}$ and therefore the outward normal vector is defined as $\vec{n}(x) = \nabla \| x \|_p^p / \| \nabla \| x \|_p^p \|$, where $\nabla \| x \|_p^p = p\, [ \sign(x_1) | x_1 |^{p-1} \ \cdots \ \sign(x_n) | x_n |^{p-1} ]^\top$.
    Since $0 \not\in \partial \mathcal{D}$, then we get that $\nabla \| x \|_p^p \neq 0$ on $\partial \mathcal{D}$ and $\vec{n}$ is a smooth vector on $\partial \mathcal{D}$.
\end{remark}

The notion of RoA also has to be adapted. Indeed, since the computations are performed in the computational domain $\mathcal{D}$, which is bounded, \textit{we can only infer properties on trajectories that do not leave the set}. Therefore, we are interested in characterizing the largest RoA forward-invariant in $\mathcal{D}$, introduced thereafter.

\begin{definition}\label{def:region-attraction}
    The largest RoA in $\mathcal{D}$, denoted $\mathcal{R}_{\mathcal{D}}^*$, is called the \textbf{maximal region of attraction in $\mathcal{D}$} and is defined as:
    \[
        \mathcal{R}_{\mathcal{D}}^* = \left\{ x_0 \in \mathcal{D}^\circ \ \Bigg| \
            \begin{array}{l} \forall t>0, \ \pi(x_0, t) \in \mathcal{D}^\circ, \\ 
        \displaystyle \lim_{t \to \infty} \| \pi(x_0, t) \| = 0,
        \end{array}
        \right\}.
    \]
\end{definition}

\begin{remark}
    Unlike the classical definition \citep[Definition~3.2]{khalil2002nonlinear}, trajectories must not leave $\mathcal{D}$. It is the largest \emph{open} set in $\mathcal{D}$ that is forward-invariant and such that the trajectories converge to the origin. However, there may exist larger forward-invariant sets that contain points on the boundary of $\mathcal{D}$.
\end{remark}

The maximal RoA in $\mathcal{D}$ contains only initial points for which the trajectory is going to the origin, and therefore is included in the maximal RoA, but is also limited to $\mathcal{D}$ by definition, leading to $\mathcal{R}^*_\mathcal{D} \subseteq \mathcal{R}^* \cap \mathcal{D}^\circ$. In the scenario that $\mathcal{R}^* \subset \mathcal{D}$, then $\mathcal{R}^* = \mathcal{R}^*_\mathcal{D}$ because for any initial condition within $\mathcal{R}^*$, the trajectory will never leave $\mathcal{D}$. Otherwise, if $\mathcal{R}^*$ does not intersect with the boundaries of $\mathcal{D}$ with an outflow, then there are no trajectories starting in $\mathcal{R}^*$ that are touching the boundary of $\mathcal{D}$, and therefore are leaving $\mathcal{D}^\circ$. Consequently, we get $\mathcal{R}^*_\mathcal{D} = \mathcal{R}^* \cap \mathcal{D}^\circ$. Otherwise, there are trajectories starting in $\mathcal{R}^*$ that are leaving $\mathcal{D}^\circ$, and we only get the strict inclusion.

Since no computation is done outside $\mathcal{D}$, we cannot certify that trajectories leaving $\mathcal{D}$ converge to the origin, even if they eventually do. This justifies the definition of RoA in $\mathcal{D}$, which is the largest RoA that we can certify by computations in $\mathcal{D}$.

\subsection{Problem formulation}

The maximal RoA for~\eqref{eq:dynamical_system} is characterized by the 1-level set of a maximal Lyapunov function, but no equivalent formulation exists for a RoA in $\mathcal{D}$. We therefore need to adapt the Zubov theorem to a computational domain.

\begin{problem}
    We want to define a maximal Lyapunov function within the computational domain $\mathcal{D}$ and adapt the Zubov theorem to that case.
\end{problem}

Moreover, computing a maximal Lyapunov function is difficult, so we approximate it with a neural network and learn it using physics-informed learning, as in \citet{liu2023physics}. A second problem of interest is to define a surrogate that approximates a maximal Lyapunov function while guaranteeing certain properties. 

\begin{problem}
    We want to construct a set of parametric functions $\mathcal{V}_N \subset C^{\infty}$ (where $N$ is the number of parameters of the neural network) such that an element of $\mathcal{V}_N$ can be pointwise arbitrarily close to a maximal Lyapunov function on $\mathcal{D}$.
\end{problem}

Problem~1, addressed in section~\ref{sec:zubov_constrained}, is a representation problem: how to generate a partial differential equation whose solution will characterize the maximal RoA in $\mathcal{D}$. Problem~2, addressed in section~\ref{sec:neural-lyapunov}, is a learning problem and concerns the universal approximation of a maximal Lyapunov function on $\mathcal{D}$ with practical considerations.

\section{Stability theorems on a computational domain}
\label{sec:zubov_constrained}

Similarly to the conventional case, we first define a Lyapunov function on a computational domain $\mathcal{D}$, and then derive the Zubov-constrained version to estimate the maximal RoA in $\mathcal{D}$.

\subsection{Lyapunov functions on $\mathcal{D}$}

Based on the original work \citep{nagumo1942lage}, it appears that adapting the definition of a Lyapunov function to work on a compact set only requires the additional property that no trajectory within its 1-level set is exiting $\mathcal{D}$. This leads to the following adapted definitions.

\begin{definition}
    We define the exit set $\Gamma$ of $\mathcal{D}$, the set of boundary points of $\mathcal{D}$ with a flow tangent or transverse to the boundary, i.e., $\Gamma = \{ x \in \partial \mathcal{D} \mid f(x) \cdot \vec{n}(x) \geq 0 \}$, where $\vec{n}(x)$ denotes the outward normal vector to $\mathcal{D}$ at the boundary point $x$ on $\partial \mathcal{D}$.
\end{definition}

First, note that since $\vec{n}$ and $f$ are continuous (see Remark~\ref{rem:continuous}), then $\Gamma$ is compact. Now we can define a Lyapunov function on $\mathcal{D}$.

\begin{definition} \label{def:lyapunov-function}
    A \textbf{Lyapunov function on $\mathcal{D}$} for $f$ is a Lyapunov function defined on $\mathcal{D}$ instead of $\R^n$ (see~\eqref{eq:lyapunov_definition}), with $d = 1$, and the additional property
    \begin{equation}
        \label{eq:boundary} 
        V(x) \geq 1 \text{ for all } x \in \Gamma,
    \end{equation}
     We denote by $\mathcal{V}_\mathcal{D}$ the set of all Lyapunov functions on $\mathcal{D}$ (fixing $d = 1$ only amounts to rescaling $V$).
\end{definition}

\begin{proposition}
    Under Assumption~\ref{ass:linearization}, the set $\mathcal{V}_\mathcal{D} \cap C^\infty$ is not empty.
\end{proposition}
\begin{proof}
    By Assumption~\ref{ass:linearization}, there exists a quadratic Lyapunov function $V_q$ and a strictly positive real number $d$ such that the set $\mathcal{R}(V_q, d)$ is a RoA around the origin.
    If $\Gamma = \emptyset$, then $V_q/d$ is a Lyapunov function on $\mathcal{D}$. Otherwise, since $\Gamma$ is a compact set, therefore, $V_q$ is lower bounded by $M_v = \min_{x \in \Gamma} V_q(x)$ on $\Gamma$. Since $V_q$ is positive definite and $0 \not\in \Gamma$, then $M_V > 0$. Consider the scaled Lyapunov function $\tilde{V}_q: x \mapsto V_q(x) \max(d, M_V)^{-1}$. This function is still a Lyapunov function with a RoA $\mathcal{R}(\tilde{V}_q, 1)$ and $\tilde{V}_q \geq 1$ on $\Gamma$. Therefore, $\tilde{V}_q \in C^\infty$ is a Lyapunov function on $\mathcal{D}$.
\end{proof}

The added condition \eqref{eq:boundary} ensures that the $1$-level set of $V$ is an invariant set within $\mathcal{D}$. In short, it means that the points on the boundary $\partial \mathcal{D}$ where the flow enters the domain $\mathcal{D}$ can be included in the closure of the RoA in $\mathcal{D}$ while others are discarded. This is verified by the following theorem.

\begin{theorem}\label{thm:basin-attraction}
    If there exists a Lyapunov function $V$ on $\mathcal{D}$, as characterized by Definition~\ref{def:lyapunov-function}, then i) the origin is asymptotically stable; ii) the set $\mathcal{R}_{\mathcal{D}}(V) = \left\{ x \in \mathcal{D}^{\circ} \mid V(x) < 1 \right\} = \mathcal{R}(V, 1) \cap \mathcal{D}^{\circ}$ is a RoA in $\mathcal{D}$.
\end{theorem}

\begin{proof}
    Since conditions~\eqref{eq:definite}-\eqref{eq:decreasing} define a local Lyapunov function, then the Lyapunov direct method \citep[Theorem~4.1]{khalil2002nonlinear} yields \textit{i)}.
    For the second point, $\mathcal{R}_{\mathcal{D}}(V)$ is contained in a level set of the Lyapunov function $V$ and will create a RoA in $\mathcal{D}$, provided that it is also forward-invariant. Assume that $\mathcal{R}_{\mathcal{D}}(V)$ is not forward-invariant in $\mathcal{D}$, then there exist $x_0 \in \mathcal{D}$ and the first-exit time $T > 0$ such that $x(T) \in \partial{D}$ with $f(x(T)) \cdot \vec{n}(x(T)) \geq 0$ \citep{nagumo1942lage}. Since $\forall t \in [0, T)$, $\pi(x_0, t) \in \mathcal{D}$, then $V(x(T)) < V(x_0) < 1$ because of the conditions \eqref{eq:definite}-\eqref{eq:decreasing}. Since $x(T) \in \partial \mathcal{D}$ with $f(x(T)) \cdot \vec{n}(x(T)) \geq 0$, by equation~\ref{eq:boundary}, we should get $V(x(T)) \geq 1$, which is a contradiction. The set $\mathcal{R}_{\mathcal{D}}(V)$ is therefore forward-invariant in $\mathcal{D}$ and $\mathcal{R}_{\mathcal{D}}(V)$ is a RoA in $\mathcal{D}$. 
\end{proof}

We are interested in estimating $\mathcal{R}_{\mathcal{D}}^*$, the largest forward-invariant RoA included in $\mathcal{D}$. To do that, we will find a Lyapunov function on $\mathcal{D}$ that admits $\mathcal{R}_{\mathcal{D}}^*$ as a level set.

\begin{definition}
    A \textbf{maximal Lyapunov function $V$ on $\mathcal{D}$} is a Lyapunov function on $\mathcal{D}$ such that $\mathcal{R}_{\mathcal{D}}(V) = \mathcal{R}_{\mathcal{D}}^*$.
\end{definition}

\subsection{Zubov theorem on computational domain}

Characterization of a maximal Lyapunov function on $\mathcal{D}$ is not provided by the classical Zubov's theorem. In this section, we derive an equivalent of Theorem~\ref{thm:zubov} to characterize the maximal RoA in $\mathcal{D}$.

The main modification needed is to ensure that $W$ from the classical Zubov's theorem is a Lyapunov function on $\mathcal{D}$ and therefore, by Theorem~\ref{thm:basin-attraction}, the intersection between $1$-level set of $V_m$ and $\mathcal{D}^\circ$ will be the maximal RoA in $\mathcal{D}$. Since we want to force $V_m$ to be $1$ on the set $\Gamma$, we can leverage concepts such as Lyapunov barrier functions \citep{tee2009barrier} and Zubov control barrier function \citep{meng2026characterization,meng2026towards} to ensure that property. To this end, we introduce a smooth function that measures the distance to the set $\Gamma$.

\begin{proposition}
    There exists a function $d_\Gamma \in C^1(\mathcal{D}, [0, \infty))$, 1-Lipschitz, and such that $\Gamma = d_{\Gamma}^{-1}(\{0\})$.
\end{proposition}
\begin{proof}
    By Theorem~2.29 in \citet{lee_introduction_2012}, since $\mathcal{D}$ is a smooth manifold and $\Gamma$ is a closed subset of $\mathcal{D}$, there exists $\kappa \in C^1(\mathcal{D}, [0, \infty))$ such that $\Gamma = \kappa^{-1}(\{0\})$. Compactness of $\mathcal{D}$ ensures that $M_\kappa = \max_{x \in \mathcal{D}} \| \nabla \kappa(x) \| < \infty$. The function $d_\Gamma = M_\kappa^{-1}\kappa$ satisfies the criteria of the proposition.
\end{proof}

\begin{lemma}
    Let $\phi: \mathcal{D} \to \R$ be a continuous positive definite function on $\mathcal{D}$ such that $\phi(x) = O(\|x\|^2)$. 
    The function 
        \begin{equation}
        \label{eq:W}
        W(x) \triangleq \left\{ \begin{array}{ll} 
            1 - \exp\left( -\int_0^\infty \omega( \pi(x, t) ) dt \right), & \text{ if } \tau(x) = \infty, \\
            1, & \text{ otherwise,}
        \end{array}\right.
    \end{equation}
    where $\omega(x) = \phi(x) d_\Gamma(x)^{-1}$ and $\tau(x) = \inf \{ t \geq 0 \mid \pi(x, t) \not\in \mathcal{D}^\circ \} \in (0, \infty]$, is well-defined on $\mathcal{D}$, continuous, and 
    \begin{equation}
        \label{eq:ROA_inclusion}
        \mathcal{R}(W, 1) \cap \mathcal{D}^\circ = \mathcal{R}_{\mathcal{D}}^*.
    \end{equation}
\end{lemma}

\begin{proof} Note first that the function $\tau$ then defines the first time the trajectory starting at $x$ exits $\mathcal{D}^\circ$.

\textit{Wellposedness on $\mathcal{D}^\circ$:}  For $x \in \mathcal{R}_\mathcal{D}^*$, then $\tau(x) = \infty$ by definition and since $0 \not \in \Gamma$, and $\Gamma$ is closed then there exists $\delta > 0$ such that $d_\Gamma(\pi(x, t)) \geq \delta$ for all $t \geq 0$. By local exponential stability, we get that there exist $C \geq 1$ and $\lambda > 0$, such that $\| \pi(x, t) \| \leq C \exp(-\lambda t) \| x \|$. Since $\phi$ is dominated by the square norm, there exist $K > 0$ and $T > 0$ such that for $t > T$, we get $|\phi(\pi(x,t))| \leq K \| \pi(x,t) \|^2$ and therefore $| \omega(\pi(x, t)) | \leq K \delta^{-1} C^2 \exp(-2\lambda t) \| x \|^2$. Since the map $t \mapsto \omega(\pi(x,t))$ is continuous on $[0, T]$, it is upperbounded by a constant for $t \leq T$ and a converging exponential for $t \geq T$. Hence, it is integrable on $[0, \infty)$ and $W$ exists and is finite.

Let $x \in \mathcal{D}^\circ \setminus \mathcal{R}_\mathcal{D}^*$. In the case $\tau(x) < \infty$, the trajectory is leaving the set $\mathcal{D}^\circ$ and therefore $W(x) = 1$ by definition. In the second case, $\tau(x) = \infty$, the trajectory never reaches the set $\partial \mathcal{D}$ so there exists $\delta > 0$ such that $d_\Gamma(\pi(x,t)) \geq \delta$. In that case, the integral $\int_0^\infty \omega( \pi(x, t) ) dt$ is indefinite and, by convention, $\int_0^\infty \omega( \pi(x, t) ) dt = \infty$, leading to $W(x) = 1$. Therefore, we get that $W$ is well defined on $\mathcal{D}^\circ$, positive definite, and \eqref{eq:ROA_inclusion} holds.

\textit{Continuity of $W$}: As shown before, $W$ is continuous on $\mathcal{R}_\mathcal{D}^*$ and on $\mathcal{D}^\circ \setminus \mathcal{R}_\mathcal{D}^*$. To show that it is continuous on $\mathcal{D}^\circ$, we therefore need to show that for any sequence $\{x_k\}_k$ of points in $\mathcal{R}_\mathcal{D}^*$ converging to $x \in \partial \mathcal{R}_\mathcal{D}^* \cap \mathcal{D}^\circ$, we also have $W(x_k) \to W(x)$.

In the case $\tau(x) = \infty$, then continuity of $\omega$ and of the map $\pi$ on the initial condition \citep[Corollary~6.1]{Sideris2013} will give $W(x_k) \to W(x)$ (this is the standard maximal Lyapunov function argument \citep[Theorem~2]{vannelli1985maximal}). The crucial case is for $\tau(x) < \infty$ and $\tau(x_k) = \infty$ where a discontinuity might appear. Note first that the trajectory must exit through $\Gamma$ by Nagumo's theorem \citep{nagumo1942lage}, consequently, $d_\Gamma(\pi(x, \tau(x))) = 0$. By continuity of the flow map on the initial condition, and continuity of $d_\Gamma$, we also get $\lim_{k \to \infty} d_\Gamma(\pi(x_k, \tau(x))) = 0$. Since the sequence $\{x_k\}_k$ converges to $x \in \partial \mathcal{R}_\mathcal{D}^* \cap \mathcal{D}^\circ$, there exists $K \in \mathbb{N}$ and $\delta > 0$ such that $\phi(\pi(x_k, t)) \geq \delta$ for all $t \in [0, \tau(x)]$ and $k \geq K$. Therefore, we get:
\begin{equation}
    \label{eq:inequality_I_1}
    \int_0^{\tau(x)} \frac{\phi(\pi(x_k, t))}{d_\Gamma(\pi(x_k, t))} dt \geq \int_0^{\tau(x)} \frac{\delta}{d_\Gamma(\pi(x_k, t))} dt.
\end{equation}
Since $d_\Gamma$ is 1-Lipschitz, it yields $d_\Gamma(\pi(x_k, t)) \leq d_\Gamma(\pi(x_k, \tau(x))) + \| \pi(x, t) - \pi(x, \tau(x)) \|$. Since $|\dot{\pi}(x, \cdot)| = |f(\pi(x, \cdot))|$, $\pi(x, \cdot)$ is $M_f$-Lipschitz on $\mathcal{D}$ for $M_f = \max_{x \in \mathcal{D}} \| f(x) \|$, so $d_\Gamma(\pi(x_k, t)) \leq d_\Gamma(\pi(x_k, \tau(x))) + M_f |t - \tau(x)|$. Plugging back into \eqref{eq:inequality_I_1} yields:
\begin{multline*}
    \int_0^{\tau(x)} \frac{\phi(\pi(x_k, t))}{d_\Gamma(\pi(x_k, t)} dt \\
    \geq \int_0^{\tau(x)} \frac{\delta}{d_\Gamma(\pi(x_k, \tau(x))) + M_f (\tau(x) -t)} dt \\
    \geq \frac{\delta}{M_f} \log\left( \frac{d_\Gamma(\pi(x_k, \tau(x))) + M_f \tau(x)}{d_\Gamma(\pi(x_k, \tau(x)))} \right) \xrightarrow[k \to \infty]{} \infty.
\end{multline*}
Therefore, we get $\int_0^{\tau(x)} \omega(\pi(x_k, t)) dt \to \infty$ and $W(x_k) \to 1 = W(x)$. 

Hence, the function $W$ is continuous on $\mathcal{D}^\circ$. Using the same arguments for $x \in \partial \mathcal{D}$, we can conclude that there exists a continuous extension of $W$ on $\mathcal{D}$ and it is therefore well-defined and continuous on $\mathcal{D}$.
\end{proof}

The previous theorem yields a function whose $1$-level set is the RoA in $\mathcal{D}$. However, computing it is rather difficult, especially since $d_\Gamma$ is unknown, and we need to integrate over a trajectory. We can instead provide a partial differential equation characterization of $W$.

\begin{theorem}
    \label{thm:zubov_constrained}
    Let $\phi: \mathcal{D} \to \R$ be a continuous positive definite function on $\mathcal{D}$ such that $\phi(x) = O(\|x\|^2)$. 
    
    The function $W: \mathcal{D} \to [0, 1]$ in \eqref{eq:W} is the unique continuous positive definite solution to
    \begin{equation}
        \label{eq:negative_zubov_D}
        \forall x \in \mathcal{D}^\circ, \quad \nabla W(x) \cdot f(x) = - \frac{\phi(x)}{d_\Gamma(x)} (1 - W(x))
    \end{equation}
    with $\left. W \right|_\Gamma = 1$. Hence, $W$ is a maximal Lyapunov function on $\mathcal{D}$.
\end{theorem}

\begin{proof} We proceed in two steps.

\textit{Solution to \eqref{eq:negative_zubov_D}:} Let $x \in \mathcal{R}_\mathcal{D}^*$ and $I(x) = \int_0^\infty \omega(\pi(x, t)) dt$. Since $t \mapsto \| \pi(x, t) \|^2$ is dominated by a decreasing exponential, it exists, and we get that:
\begin{multline*}
    \frac{dI(\pi(x,\cdot))}{dt}(0) = \nabla W(x) \cdot f(x) \\
    \qquad = \lim_{h \to 0_+} \frac{I(\pi(x,h)) - I(x)}{h} \\
    = - \tfrac{1}{h} \int_0^h \omega(\pi(x,t+s)) ds = - \omega(x).
\end{multline*}
Since $W = 1 - \exp(-I)$, we get \eqref{eq:negative_zubov_D}. 
On the set $\mathcal{D}^\circ \setminus \mathcal{R}_\mathcal{D}^*$, $W$ is constant and equal to $1$, therefore $\nabla W = 0$ where it exists and \eqref{eq:negative_zubov_D} holds. 

\textit{Uniqueness:} Pick $x \in \mathcal{D}^\circ$. Let $\tilde{W}$ be any continuous positive definite solution to \eqref{eq:negative_zubov_D} with $\left. \tilde{W} \right|_\Gamma = 1$. Define $z(t) \triangleq 1 - \tilde{W}(\pi(x, t))$. We get $\dot{z}(t) = -\nabla \tilde{W}(\pi(x,t)) \cdot f(\pi(x,t)) = \omega(\pi(x,t)) z(t)$, which is a linear ordinary differential equation whose unique solution is $z(t) = z(0) \exp(\int_0^t \omega(\pi(x,s)) ds )$ as long as $\pi(x, s) \in \mathcal{D}^\circ$ for $s \in [0, t]$. Since $\omega \geq 0$, we have $|z(t)| \geq |z(0)|$.

For $x \in \mathcal{R}_\mathcal{D}^*$, the integral $I(x) = \int_0^\infty \omega(\pi(x,s)) ds$ is finite (see the proof of the previous lemma) and $\pi(x, t) \xrightarrow[t\to \infty]{} 0$, so by continuity and definiteness of $\tilde{W}$, we get $z(t) \xrightarrow[t \to \infty]{} 1$. It yields $1 = z(0) \exp(I(x))$, i.e., $\tilde{W}(x) = 1 - \exp(-I(x)) = W(x)$.

For $x \in \mathcal{D}^\circ \setminus \mathcal{R}_\mathcal{D}^*$, there are two cases. If $\tau(x) = \infty$, then, as in the previous lemma, $\int_0^t \omega(\pi(x,s)) ds \xrightarrow[t \to \infty]{} \infty$. Since $\tilde{W}$ is continuous on the compact set $\mathcal{D}$ and $z$ is bounded, this is only possible if $z(0) = 0$, i.e., $\tilde{W}(x) = 1$. If $\tau(x) < \infty$, the trajectory exits through $\Gamma$ \citep{nagumo1942lage}, so by continuity of $\tilde{W}$ and $\left. \tilde{W} \right|_\Gamma = 1$, $z(t) \to 0$ as $t \to \tau(x)^-$. Since $|z(0)| \leq |z(t)|$ for $t < \tau(x)$, we get $z(0) = 0$ and $\tilde{W}(x) = 1$.

We get that $\tilde{W} = W$ on $\mathcal{D}^\circ$ and $\tilde{W}$ can be extended by continuity to $\mathcal{D}$.
\end{proof}

We conclude this section with the regularity of $W$, needed for its neural approximation.

\begin{theorem}
    \label{theo:regularity_W}
    Under Assumption~\ref{ass:linearization} with $f$ continuously differentiable, there exists a continuous positive definite function $\phi$, dominated by the square norm at the origin, such that $W$ as in \eqref{eq:W} is $C^4$ in a neighborhood of the origin, and $C^1$ on any compact set included in $\mathcal{R}_\mathcal{D}^*$.
\end{theorem}

\begin{proof}
    We follow a similar reasoning as in the proof of \citet[Lemma~4]{vannelli1985maximal}. Let $V_q(x) = x^{\top} P x$ such that $A^\top P + PA = -Q$ where $P, Q \succ 0$. This is possible by Assumption~\ref{ass:linearization}. Then there exists a closed ball $\mathcal{B}_{\delta} \subset \mathcal{R}_\mathcal{D}^*$ of radius $\delta > 0$ around the origin in which $\nabla V_q(x) \cdot f(x) < 0$. We define $\phi = - d_\Gamma \nabla V_q \cdot f$ on $\mathcal{B}_{\delta}$ which is positive definite. 
    
    Since $\mathcal{B}_{\delta}$ is a closed ball with a strictly positive radius, there exists another closed ball $\mathcal{B}_{\delta_0}$ of radius $\delta_0 \in (0, \delta)$. By the cutoff lemma \citep[Lemma~2.21]{lee_introduction_2012}, there exists a cutoff function $\eta$ which is $C^1(\mathcal{D}, [0, 1])$ such that $\left. \eta \right|_{\mathcal{B}_{\delta_0}} = 1$ and $\left. \eta \right|_{\mathcal{D} \setminus \mathcal{B}_{\delta}} = 0$. We can define an extension of $\phi$ to $\mathcal{D}^\circ$ by:
    \[
        \phi_e(x) = \phi(x) \eta(\|x\|) + (1-\eta(\|x\|))\frac{\|x\|}{\delta} \phi\left( \frac{\delta x}{\| x \|} \right)
    \]
    which is $C^1(\mathcal{D}, \R)$, $\phi_e(x) = d_\gamma(x) x^{\top} Q x$ in $\mathcal{B}_{\delta_0}$ and therefore $\phi_e(x) = O(\|x\|^2)$. Moreover, it is positive definite on $\mathcal{D}$. 
    
    For $\omega = \phi_e d_\Gamma^{-1} \in C^1(\mathcal{D}^\circ, \R)$, we can define $W$ as in \eqref{eq:W} and $W = 1 - \exp(V_q)$ in $\mathcal{B}_{\delta_0}$ by the uniqueness property. With this choice of $\phi$, $W$ is therefore $C^4$ in a neighborhood of the origin and, by the previous theorem, $W$ is also continuous on $\mathcal{D}$. We now need to prove that $W \in C^1(K, \R)$ for any compact set $K \subset \mathcal{R}_\mathcal{D}^*$.
    
    First, note that, with the previous construction, we get $\phi_e \in C^1(\mathcal{D}, \R)$. Let a compact set $K \subset \mathcal{R}_\mathcal{D}^*$, by Corollary~6.1 from \citet{Sideris2013}, we get that $x \mapsto \pi(x, t)$ has the same regularity as $f$ and is therefore $C^1(K, K_2)$ for any $t \in [0, \infty)$ and $K_2 = \{ \pi(x_0, t) \mid (x_0, t) \in K \times [0,\infty) \} \cup \{ 0 \}$. Since $K$ is compact, by continuity of $\pi$ on the initial condition, and since $\pi(\cdot, t) \xrightarrow[t\to\infty]{} 0$ on $K$, the set $K_2$ is also compact. Since $d_\Gamma$ is continuous strictly positive on $\mathcal{D} \setminus \Gamma$ and $\Gamma \cap K_2 = \emptyset$, then $x \mapsto \omega(\pi(x,t))$ is $C^1(K, \R)$. By the chain rule, we therefore get
$\nabla_x h(x,t) = \nabla_x \pi(x,t) \nabla \omega(\pi(x,t))$ for all $x \in K$, $t \geq 0$.
    By Lemma~4 in \citet{harting2026how}, $\nabla_x \pi(x,\cdot)$ converges exponentially to $0$, that is, there exist $C, \lambda > 0$ such that $\| \nabla_x \pi(x,t) \| \leq C \exp(-\lambda t) \| x \|$ for all $x \in K$ and $t \geq 0$. Since $K$ is compact, $C_2 = C ( \max_{x \in K} \| x \| ) ( \max_{x \in K_2} \| \nabla \omega(x) \| ) < \infty$, and it yields $\| \nabla_x h(x,t) \| \leq C_2 \exp(-\lambda t) \triangleq \Theta(t)$ for all $x \in K$, $t \geq 0$. Since $\Theta$ is integrable over $[0, \infty)$, then, by the dominated convergence theorem, we get that $\nabla I(x) = \int_0^\infty \nabla_x h(x,t) dt$, which is continuous in $x$, concluding the proof.
\end{proof}

\section{Universal Approximations of Lyapunov Functions}
\label{sec:neural-lyapunov}

In this section, we will construct the set of functions $\mathcal{V}_N$ and provide its universal approximation capabilities.

\subsection{Architecture design}

The purpose of this subsection is to design a novel neural network architecture such that the neural Lyapunov function has local guarantees around the origin. The key idea is then to enforce the local behavior of the neural Lyapunov function using a Taylor expansion. Before moving on, we make the following assumption on the regularity of the target Lyapunov function for this analysis.

\begin{assumption}
    \label{ass:regularity_V}
    The target Lyapunov functions on $\mathcal{D}$ are $C^4$ on an open ball $\mathcal{B}$ around the origin, $C(\mathcal{D}, \R)$, and $C^1$ on any compact in $\mathcal{R}_D(V)$.
\end{assumption}

A maximal Lyapunov function with the mentioned regularity property exists from Theorem~\ref{theo:regularity_W}. 
The high regularity around the origin enables us to use a Taylor expansion with integral remainder \citep[Theorem~5.4]{coleman2012calculus}, leading to, for any $x \in \mathcal{B}$:
\begin{multline} \label{eq:V}
    V(x) = V(0) + x^{\top} \nabla V(0) + \tfrac{1}{2} x^{\top} \nabla^2V(0) x + R(x),
\end{multline}
where $\nabla V(0)$ and $\nabla^2 V(0)$ are the gradient and the Hessian of $V$ at the origin with
\begin{equation}
    \label{eq:R}
    R(x) = \sum_{i \in \mathcal{I}_n} \int_0^1 \frac{3(1-t)^2}{i_1! \cdots i_n!} \frac{\partial^3 V}{\partial x_1^{i_1} \cdots \partial x_n^{i_n}}(t \cdot x) dt \prod_{k=1}^{n} x_k^{i_k},
\end{equation}
and $\mathcal{I}_n = \{ (i_1, \dots, i_n) \in \{0, 1, 2, 3\}^{n} \mid \sum_{k=1}^n i_k = 3 \}$, whose cardinal is $m_n = \binom{n+2}{3}$.

Since $\nabla^3 V \in C^1$, for any closed ball $\bar{\mathcal{B}} \subset \mathcal{D}$, we get that $\sup_{x \in \bar{\mathcal{B}}} \| \nabla^3 V(x) \| < \infty$ and therefore $R(x) = O(\| x \|^3)$.

We continue our architectural design by simplifying \eqref{eq:V} using the following lemma (based on the indirect Lyapunov method).

\begin{lemma}\label{lem:lemma1}
    Let $V$ be a Lyapunov function for~\eqref{eq:dynamical_system} on $\mathcal{D}$; under Assumptions~\ref{ass:linearization}-\ref{ass:regularity_V} we have that:
    \begin{enumerate}
        \item $V(0) = \nabla V(0) = 0$ and $\nabla^2 V(0)$ is a symmetric positive semi-definite matrix.
        \item $A^{\top} \nabla^2 V(0) + \nabla^2 V(0) A \preceq 0$.
    \end{enumerate}
\end{lemma}
The proof is based on equivalence relations at the origin and follows classical arguments as in \citet{khalil2002nonlinear}.

Applying Lemma~\ref{lem:lemma1} to~\eqref{eq:V} yields:
\begin{equation} \label{eq:Vstar_extended}
    \forall x \in \mathcal{B}, \quad V(x) = \frac{1}{2} x^{\top} \nabla^2V(0) x + R(x).
\end{equation}

The previous equation shows that if we want to approximate $V$, we need to choose a positive semi-definite matrix $P = \nabla^2 V(0)$ and approximate $R$. However, $R$ is only defined in $\mathcal{B}$, and if we use the universal approximation theorem directly on $R$, we get a parametric approximation $\hat{R}$, but $\hat{R}$ is not necessarily dominated by $\|\cdot\|^3$, which is required for local guarantees around the origin.

We therefore seek a representation of $R$ on $\mathcal{D}$ in which the cubic vanishing at the origin is enforced structurally, rather than learned by the neural network.

\begin{lemma}
    \label{lem:hadamard}
    Under Assumption~\ref{ass:regularity_V}, there exist $\{R_i\}_{i \in \mathcal{I}_n}$ which are $C(\mathcal{D}, \R) \cap C^1(\mathcal{R}_\mathcal{D}(V), \R)$ functions such that \eqref{eq:V} holds with:
    \begin{equation}
        \label{eq:R_ext}
        R(x) = \sum_{i \in \mathcal{I}_n} R_i(x) \prod_{k=1}^n x_i^{i_k} = O(\|x\|^3).
    \end{equation}
\end{lemma}

\begin{proof}
    Since $\mathcal{B}$ is a non-empty open ball around the origin in which $V$ is $C^4$, there exists a closed ball $\tilde{\mathcal{B}}$ strictly included in $\mathcal{B}$. By the extension lemma \citep[Lemma~2.26]{lee_introduction_2012}, there exists a function $\chi \in C^1(\mathcal{D}, [0, 1])$ such that $\left. \chi \right|_{\tilde{\mathcal{B}}} = 1$ and $\left. \chi \right|_{\mathcal{D} \setminus \mathcal{B}} = 0$. 

    Let, for $i \in \mathcal{I}_n$, $R_i^{\text{near}}$ be the $i$-th integral term in \eqref{eq:R} (without the monomial), and $\tilde{R}(x) = V(x) - \tfrac{1}{2}x^{\top} \nabla^2 V(0) x$ where $\tilde{R}$ has the same regularity as $V$. Define $R^{\text{near}} = \chi \tilde{R}$ and $R^{\text{far}} = (1-\chi) \tilde{R}$. Then we get $R^{\text{near}}(x) = \sum_{i \in \mathcal{I}_n} \chi(x)R_{i}^{\text{near}}(x) \prod_{k=1}^n x_i^{i_k}$ for all $x \in \mathcal{B}$.
    Since $\chi$ is $C^1$ and is zero outside $\mathcal{B}$, then $\chi R^{\text{near}}_{i}$ has an extension by zero to $\mathcal{D}$ which is continuous on $\mathcal{D}$ and $C^1$ on $\mathcal{R}_\mathcal{D}(V)$ with $\tilde{R} = R^{\text{near}} + R^{\text{far}}$.

    Let $q(x) = \sum_{i \in \mathcal{I}_n} \left( \prod_{k=1}^n x_k^{i_k} \right)^2$. We get that $q(x) > 0$ on $\mathcal{D} \setminus \{0\}$. Therefore, we can define 
    \[
        S_{i}(x) = \left\{ \begin{array}{ll}
        \frac{\prod_{k=1}^n x_k^{i_k}}{q(x)} R^{\text{far}}(x) & \text{ if } x\neq 0, \\
        0 & \text{ if } x = 0.
        \end{array}\right.
    \]
    Since $R^{\text{far}} \in C^1(\mathcal{R}_\mathcal{D}(V), R)$ and is identically zero in $\tilde{\mathcal{B}}$ then $S_{i}$ is continuous on $\mathcal{D}$ and $C^1(\mathcal{R}_\mathcal{D}(V), \R)$. It yields $\sum_{i \in \mathcal{I}_n} S_{i}(x) \prod_{k=1}^n x_i^{i_k} = R^{\text{far}}(x) q(x)^{-1} q(x) = R^{\text{far}}(x)$.
    Consequently, for all $x \in \mathcal{D}$ with
    \[
        R_{i}(x) = \chi(x)R_{i}^{\text{near}}(x) + S_{i}(x)
    \]
    we get \eqref{eq:R_ext} with $R_{i} \in C(\mathcal{D}, \R) \cap C^1(\mathcal{R}_\mathcal{D}(V), \R)$. Since these functions are continuous on $\mathcal{D}$, we get that they are upperbounded by $K$ on $\mathcal{D}$ and $|R(x)| \leq K \prod_{k=1}^n |x_i|^{i_k} \leq K \|x\|^3$.
\end{proof}

We can now approximate each $R_{i}$ by a neural network $\hat{R}_{i}$, defined on the whole domain $\mathcal{D}$ and such that $\hat{R}$ is still dominated by $\| \cdot \|^3$. This is the purpose of the two definitions that follow.

\begin{definition}\label{def:neural-residual}
    We define the \textbf{neural network remainders} $\hat{R}_i \in C^4$, with $i \in \mathcal{I}$ and $N > 0$ neurons per layer, as
    \begin{equation} \label{eq:nn_residual}
        \hat{R}_i(x) = \textsc{MLP}_N(x) \cdot e_i
    \end{equation}
    with $\{e_i\}_{i \in \mathcal{I}_n}$ a normed basis of $\mathbb{R}^{m_n}$, and
    \[
        \textsc{MLP}_N(x) = W_{N_l+1} H_{W_{N_l},b_{N_l}} \circ \cdots \circ H_{W_{0},b_{0}}(x) + b_{N_l+1}
    \]
    where $N_l \in \mathbb{N} \backslash \{0\}$ is the number of hidden layers, $W_0 \in \mathbb{R}^{N \times n}$, $W_i \in \mathbb{R}^{N \times N}$, $W_{N_l+1} \in \mathbb{R}^{m_n \times N}$ are the weights, $b_i \in \mathbb{R}^N$ are the biases, $\psi \in C^4(\mathbb{R}, \mathbb{R})$ is not polynomial, and with
    \[
        H_{W,b}(x) = \psi.(Wx + b).
    \]
\end{definition}

In the following, we denote by $\Theta = \left\{ (W_i, b_i) \right\}_{i = 0, \dots, N_l+1}$ the tensor collecting all the parameters (weights and biases) of the neural network residual. We are now ready to define the proposed architecture.

\begin{definition}[Taylor Lyapunov Network]\label{def:taylor-neural-lyapunov}
    A \textbf{Taylor Lyapunov network} is defined as
    \begin{equation}
        \label{eq:VhatFull}
        \hat{V}(x) = \frac{1}{2} x^{\top} P x + \sum_{i \in \mathcal{I}_n} \hat{R}_{i}(x) \prod_{k=1}^{n} x_k^{i_k}
    \end{equation}
    where $P \in \mathbb{R}^{n \times n}$ is a positive definite matrix such that $A^{\top}P + PA \prec 0$ and $\hat{R}_i$ as in \eqref{eq:nn_residual}.
\end{definition}

\begin{remark}
    Note that $\hat{V}$ inherits the regularity from $\hat{R}_i$ and is therefore $C^4$ on $\mathcal{D}$.
\end{remark}

The following subsection shows that the proposed architecture generates local Lyapunov functions by design and that it is a universal approximator for Lyapunov functions in $\mathcal{D}$. That is, choosing $N$ large enough, we can always find $P$ and $\Theta$ so that $\hat{V}$ approximates $V$ arbitrarily closely.

\subsection{Universal approximation theorem}

In this subsection, we want to show that the set of Taylor Lyapunov networks
\[
    \mathcal{V}_N \triangleq \left\{ \hat{V} : \mathcal{D} \to \R \mid \begin{array}{l}
        \forall \Theta, P \succ 0, A^{\top} P + PA \prec 0, \\
        \text{ Equation~\eqref{eq:VhatFull} holds}
    \end{array} \right\}
\]
is providing a solution to Problem~2.

We start with the structural local Lyapunov property of any function in $\mathcal{V}_N$.

\begin{proposition}
    \label{prop:conventional_lyap_function}
    For any $N > 0$, a function in $\mathcal{V}_N$ is a local Lyapunov function.
\end{proposition}
\begin{proof}
    Let $\hat{V} \in \mathcal{V}_N$ as in \eqref{eq:VhatFull} with $\hat V_2(x) = \tfrac{1}{2} x^{\top} P x$ and $\hat V_3(x) = \hat{V}(x) - V_2(x)$.

    \textit{i) The function $\hat{V}$ is positive definite in a neighborhood around the origin.}
    
    First, note that for $x = (x_1, \dots, x_n) \in \R^n$, $|x_k| \leq \| x \|$. Then we get $\prod_{k=1}^{n} x_k^{i_k} \leq \|x\|^{i_1+\cdots+i_n} = \| x \|^3$ for $i \in \mathcal{I}_n$. Consequently, Cauchy-Schwarz yields $\hat V_3(x) \leq m_n \max_{i\in\mathcal{I}_n} | \hat{R}_i(x) | \| x \|^3$.
    Since $\mathcal{D}$ is a compact and $\hat{R}_i$ are continuous, then $| \hat{R}_i |$ is upperbounded on $\mathcal{D}$, and $\hat V_3(x) = O(\| x \|^3)$.
    
    Leveraging the previous comparison yields that $\hat{V}$ is positive definite in a neighborhood $\mathcal{N}_1$ of the origin.

    \textit{ii) The function $\nabla \hat{V} \cdot f$ is negative definite in a neighborhood around the origin.}

    By definition, we get that $\hat V_3$ is $C^3$ and therefore, we can compute its gradient on $\R^n$ as:
    \begin{equation}
        \label{eq:nabla_V_3}
        \nabla \hat V_3(x) = \nabla \hat{V}(x) - \nabla \hat V_2(x)
    \end{equation}
    where $\nabla V_2(x) = x^{\top} \nabla^2 \hat{V}(0)$. Since $\nabla \hat{V} \in C^2$, we can apply a first-order Taylor expansion with Lagrange's remainder \citep[Theorem~5.2]{coleman2012calculus}, leading to $\nabla \hat{V}(x) = \nabla \hat{V}(0) + \nabla^2 \hat{V}(0) x + O(\| x \|^2)$. Noting that $\nabla \hat{V}(0) = 0$ by definition, plugging this back in \eqref{eq:nabla_V_3} yields $\nabla \hat V_3(x) = O(\|x\|^2)$. Finally, this leads to:
    \[
        \nabla \hat{V}(x) \cdot f(x) = \tfrac{1}{2}x^{\top} \left( A^{\top} P + PA \right) x + \| f(x) \| O(\| x \|^2).
    \]
    Leveraging Assumption~\ref{ass:linearization}, we get $\| f(x) \| = O(\| x \|)$ and therefore $\nabla \hat{V}(x) \cdot f(x) \leq - \lambda_m \| x \|^2 + O(\| x \|^3)$ where $\lambda_m$ is the smallest eigenvalue of $-\tfrac{1}{2}(A^{\top} P + PA) > 0$. Therefore, $\nabla \hat{V} \cdot f$ is negative-definite in a neighborhood $\mathcal{N}_2$ around the origin.

    \textit{iii) Conclusion}

    Since $\mathcal{N}_1 \cap \mathcal{N}_2$ is a neighborhood around the origin, there exists $d > 0$ such that $\mathcal{R}(\hat{V}, d) \subseteq \mathcal{N}_1 \cap \mathcal{N}_2$. Hence, $\hat{V}$ is a local Lyapunov function.
\end{proof}

\begin{remark}
    Note that the neural network architectures introduced in \citet{liu2023physics}, \citet{LarsGrune}, \citet{akhtar2026learning}, or \citet{gaby2022lyapunovnet} to approximate a Lyapunov function are not necessarily leading to Lyapunov functions since there is no local guarantee around the origin. This explains the concept of $\delta$-Lyapunov functions introduced in \citet[Definition~1]{gaby2022lyapunovnet} and related to ultimate boundedness \citep{khalil2002nonlinear}.
\end{remark}

We want to show that an element of $\mathcal{V}_N$ can be made as close as possible to a smooth Lyapunov function on $\mathcal{D}$; this is the $C^1$-universal approximation.

\begin{proposition}
    \label{prop:universal_approximation_V}
    Let $V$ such that Assumption~\ref{ass:regularity_V} holds. For any $\varepsilon > 0$ and compact $K \subset \mathcal{R}_\mathcal{D}(V)$, there exists $N$ and a function $\hat{V}$ in $\mathcal{V}_N$ such that:
    \begin{equation}
        \label{eq:universal_approximation_V}
        \left\{
            \begin{array}{l}
                \max_{x \in \mathcal{D}} | V(x) - \hat{V}(x) | \leq \varepsilon, \\
                \max_{x \in K} \| \nabla V(x) - \nabla \hat{V}(x) \|  \leq \varepsilon, \\
            \end{array}
        \right.
    \end{equation}
    for any compact set $K$ in $\mathcal{R}_\mathcal{D}(V)$.
\end{proposition}

\begin{proof}
    Let a quadratic local Lyapunov function $V_q(x) = \tfrac{1}{2} x^{\top} \tilde{P} x$ with $\tilde{P} \succ 0$ and $A^{\top} \tilde{P} + \tilde{P} A \prec 0$ (which always exists by Assumption~\ref{ass:linearization}) and define the function:
    \[
        \tilde{V}(x) = V(x) + \alpha V_q(x)
    \]
    with $\alpha = \tfrac{1}{2} \varepsilon \left( \max_{x \in \mathcal{D}} V_q(x) \right)^{-1} > 0$. Note that $\nabla^2 \tilde{V} = \nabla^2 V(0) + \alpha \tilde{P} \succ 0$ (by Lemma~\ref{lem:lemma1}).

    From Lemma~\ref{lem:hadamard}, there exist functions $R_i \in C(\mathcal{D}, \R) \cap C^1(\mathcal{R}_\mathcal{D}(V), \R)$ such that $V$ writes as \eqref{eq:V} with $R$ as in \eqref{eq:R_ext}. For $P \triangleq \nabla^2 \hat{V}(0) = \nabla^2 \tilde{V}(0) \succ 0$, there exists $N > 0$ such that there is a $\hat{R}_N$ (and therefore some weights and biases $\Theta$) such that:
    \begin{equation}
        \label{eq:universal_approximation_R}
        \max_{i \in \mathcal{I}_n} \max_{x \in \mathcal{D}} | R_i(x) - \hat{R}_i(x) | \leq \frac{\varepsilon}{2 m_n \max_{x \in \mathcal{D}} \| x \|^3}.
    \end{equation}
    This inequality follows from the universal approximation theorem \citep{hornik1991approximation} since $\mathcal{D}$ is a compact set and all $R_i$ are continuous and $\max_{x \in \mathcal{D}} \| x \|^3 < \infty$.

    We then get, for $x \in \mathcal{D}$:
    \begin{multline}
        \label{eq:triangle_inequality_V}
        | V(x) - \hat{V}(x) | \leq | \tilde{V}(x) - \hat{V}(x) | + |\alpha V_q(x)| \\
        \leq | \tilde{V}(x) - \hat{V}(x) | + \tfrac{1}{2} \varepsilon.
    \end{multline}
    Since, by Assumption~\ref{ass:regularity_V}, $V$ admits the decomposition in \eqref{eq:V} on $\mathcal{D}$, it yields
    $| \tilde{V}(x) - \hat{V}(x) | \leq m_n \max_{i \in \mathcal{I}_n} | R_i(x) - \hat{R}_i(x) | \| x \|^3$.
    Leveraging \eqref{eq:universal_approximation_R} yields $\tilde{V}(x) - \hat{V}(x) | \leq \frac{\varepsilon}{2}$.
    
    We obtain the first part of the inequality in \eqref{eq:universal_approximation_V} by combining the previous result with \eqref{eq:triangle_inequality_V}.

    Since $\nabla R_i \in C(\mathcal{R}_\mathcal{D}(V), \R^n)$ for $i \in \mathcal{I}$, a similar reasoning on the gradient applies using the $C^1$-universal approximation theorem from \citet{hornik1991approximation} on any compact set $K$ included in $\mathcal{R}_\mathcal{D}(V)$, leading to the result in \eqref{eq:universal_approximation_V}.
\end{proof}

Combining the two propositions from this section leads to the approximation $\hat{V} \in \mathcal{V}_N$ of $V$ to be a local Lyapunov function in a ball uniquely prescribed by the properties of $V$, as shown below.

\begin{lemma}
    \label{lem:uniform_CV}
    Let a Lyapunov function $V$ such that Assumption~\ref{ass:regularity_V} holds together with
    \begin{equation}
        \label{eq:strict_lyap_V}
        \nabla^2 V(0) \succ 0, \quad  \nabla^2 V(0)^{\top} A + A  \nabla^2 V(0) \prec 0.
    \end{equation}
    If $f$ satisfies Assumption~\ref{ass:linearization} and is $C^2$ around the origin, there exist $\delta_0 > 0$ and $\varepsilon_* > 0$, such that, for any $\varepsilon < \varepsilon_*$ and any compact set $K \subset \mathcal{R}_\mathcal{D}(V)$, \eqref{eq:universal_approximation_V} holds with:
    \begin{equation}
        \label{implication:hatV}
        \forall x \in \mathcal{D}\setminus\{0\}, \quad \| x \| < \delta_0 \Rightarrow \left\{ \begin{array}{l}
            \hat{V}(x) > 0, \\
            \nabla \hat{V}(x) \cdot f(x) < 0.
        \end{array}
        \right.
    \end{equation}
\end{lemma}

\begin{proof} The proof is technical and therefore reported in Appendix~\ref{app:uniform_CV}. \end{proof}

We have shown in Proposition~\ref{prop:conventional_lyap_function} that any function in $\mathcal{V}_N$ is a local Lyapunov function. Thanks to Proposition~\ref{prop:universal_approximation_V}, we also have arbitrary precision on the approximation of $V$. Leveraging the previous lemma, we prove that, given a target Lyapunov function $V$ on $\mathcal{D}$, there exists a function in $\mathcal{V}_N$ (for a sufficiently large $N$) which is a Lyapunov function on $\mathcal{D}$ and approximates $V$ with an arbitrary precision, which answers Problem~1.

\begin{theorem}
    \label{theo:universal_approximation_VN}
    Under Assumption~\ref{ass:linearization} with $f \in C^2$ around the origin, a function $V$ satisfying Assumption~\ref{ass:regularity_V} can be approximated by a Lyapunov function $\hat{V} \in \mathcal{V}_N$ for $N$ sufficiently large, i.e.
    \begin{multline}
        \label{eq:universal_approximation_in_V_D}
        \forall \varepsilon > 0, \exists N > 0, \exists \hat{V} \in \mathcal{V}_N \cap \mathcal{V}_\mathcal{D}, \\
        \max_{x \in \mathcal{D}} | V(x) - \hat{V}(x) | \leq \varepsilon.
    \end{multline}
\end{theorem}

\begin{proof}
    Let $\varepsilon > 0$ be arbitrary, and a Lyapunov function $V$ be as prescribed in the theorem. From Assumption~\ref{ass:linearization}, there exists $P \succ 0$ such that $V_q(x) = \tfrac{1}{2} x^{\top} P x$ is a local Lyapunov function. Consider $\beta \geq 0$, and $\tilde{V} = V + \beta V_q$. Note that $\tilde{V}$ satisfies Assumption~\ref{ass:linearization} and \eqref{eq:strict_lyap_V}. Since $\tilde V$ is continuous on $\mathcal{D}$, define $M_V = \max_{x \in \mathcal{D}} | \tilde{V}(x)|$.

    Choose $\alpha$ and $\rho$ sufficiently close to $1$ such that:
    \begin{equation}
        \label{eq:alpha_rho}
        \rho \in \left(\left(1 + \frac{\varepsilon}{3M_V} \right)^{-1}, 1 \right), \quad \alpha \in \left(\rho^{-1}, 1 + \frac{\varepsilon}{3M_V} \right).
    \end{equation}
    Therefore, it yields $\rho < 1$, $\alpha > 1$, and $\alpha \rho > 1$. Choose $\beta$ sufficiently small such that $\tilde{V}$ is a Lyapunov function on $\mathcal{D}$ with $\mathcal{R}(V, \rho) \subseteq \mathcal{R}(\tilde{V}, 1)$. Define $V_\alpha \triangleq \alpha \tilde{V}$. Since multiplication of a Lyapunov function on $\mathcal{D}$ by a scalar larger than $1$ is still a Lyapunov function on $\mathcal{D}$, then $\tilde{V}_\alpha \in \mathcal{V}_\mathcal{D}$. 

    By Proposition~\ref{prop:universal_approximation_V}, for any $\eta > 0$, and any compact set $K \subset \mathcal{R}_\mathcal{D}(\tilde{V}_\alpha)$ containing a neighborhood of the origin, there exist $N$ and $\hat{V} \in \mathcal{V}_N$ such that:
    \begin{equation}
        \label{eq:approx_V_alpha}
        \left\{
            \begin{array}{l}
                \max_{x \in \mathcal{D}} | \tilde{V}_\alpha(x) - \hat{V}(x) | \leq \eta, \\
                \max_{x \in K} \| \nabla \tilde{V}_\alpha(x) - \nabla \hat{V}(x) \| \leq \eta, \\
            \end{array}
        \right.
    \end{equation}
    We now choose $\eta$ sufficiently small to satisfy all the requirements below.

    \textit{Inclusion of RoAs:} 
    Choose $\eta$ such that $\eta < \alpha \rho - 1$. For any $x \in \mathcal{R}_{\mathcal{D}}(\hat{V})$, we get $\hat{V}(x) < 1$ and therefore $\tilde{V}_\alpha(x) \leq \hat{V}(x) + |\tilde{V}_\alpha(x) - \hat{V}(x)| < 1 + \eta < \alpha \rho$. Hence, $\tilde{V}(x) < \rho$ and we get 
    \begin{equation}
        \label{eq:set_inclusion}
        \mathcal{R}_{\mathcal{D}}(\hat{V}) \subset \mathcal{R}(\tilde{V}, \rho) \cap \mathcal{D}^\circ.
    \end{equation}

    \textit{Positive definiteness of $\hat{V}$:}
    Lemma~\ref{lem:uniform_CV} gives the existence of $\eta_0, \delta_0$ such that for $\eta < \eta_0$, we get $\hat{V}$ positive definite on $\mathcal{B}_{\delta_0}$. Outside this ball, consider the compact set $K_\rho = \{ x \in \mathcal{D} \mid \| x \| \geq \delta_0, \tilde{V}(x) \leq \rho \}$. If $K_\rho$ is empty, then $\hat{V}$ is positive definite on $\mathcal{R}(\tilde{V}, \rho) \cap \mathcal{D}^\circ$. Otherwise $K_\rho$ is a compact set and since $V$ is positive definite on $K_\rho$, we get $m_\rho = \min_{x \in K_\rho} \tilde{V}(x) > 0$. Choose $\eta < \min(\eta_0, \alpha m_\rho)$, then for $x \in K_\rho$, $\hat{V}(x) \geq \tilde{V}_\alpha(x) - \eta \geq \alpha m_\rho - \eta > 0$.

    Therefore, $\hat{V}$ is positive definite on $\mathcal{B}_{\delta_0} \cup K_\rho \supseteq \mathcal{R}(\tilde{V}, \rho) \cap \mathcal{D}^\circ$. Leveraging the set inclusion \eqref{eq:set_inclusion}, we get that $\hat{V}$ is positive definite on $\mathcal{R}_{\mathcal{D}}(\hat{V})$.

    \textit{Negative definiteness of the orbital derivatives:} 
    Since $\tilde{V}$ is a Lyapunov function on $\mathcal{D}$ and $K_\rho$ is a compact set included in $\mathcal{R}_{\mathcal{D}}(V)$ that does not contain the origin\footnote{Similarly to the previous case, if $K_\rho$ is empty then $\nabla \hat{V} \cdot f$ is negative definite on $\mathcal{B}_{\delta_0} \cup K_\rho \supseteq \mathcal{R}(V, \rho)$.}, then $\ell_\rho = - \min_{x \in K_\rho} \nabla \tilde{V}(x) \cdot f(x) > 0$. Consequently, $\nabla \tilde{V}_\alpha \cdot f \leq - \alpha \ell_\rho < 0$. Let $\displaystyle M_f = \max_{x \in \mathcal{D}} \| f(x) \|$, then, from \eqref{eq:approx_V_alpha}, we get for $x \in K_\rho$
    \begin{multline*}
        \nabla \hat{V}(x) \cdot f(x) = \nabla \tilde{V}_\alpha(x) \cdot f(x) + (\nabla \hat{V}(x) - \nabla \tilde{V}_\alpha(x)) \cdot f(x) \\
        \leq -\alpha \ell_\rho + \eta M_f.
    \end{multline*}
    Choose $\eta < \frac{\alpha \ell_\rho}{2 M_f}$, then the orbital derivatives of $\hat{V}$ are strictly negative on $K_\rho$. Since The orbital derivatives are negative definite in $\mathcal{B}_{\delta_0}$ then they are negative definite on $\mathcal{B}_{\delta_0} \cup K_\rho \supseteq \mathcal{R}(\tilde{V}, \rho) \cap \mathcal{D}^\circ$.  By the set inclusion \eqref{eq:set_inclusion}, $\nabla \hat{V} \cdot f$ is definite negative in $\mathcal{R}_{\mathcal{D}}(\hat{V})$.

    \textit{Boundary condition:}
    The last condition to check for $\hat{V}$ to be a Lyapunov function on $\mathcal{D}$ is its value on $\Gamma$. For any $x \in \Gamma$, leveraging \eqref{eq:approx_V_alpha}, we get $\hat{V}(x) \geq \tilde{V}_\alpha(x) - \eta = \alpha \tilde{V}(x) - \eta \geq \alpha - \eta$. For any $\eta < \alpha - 1$, we get $\hat{V} \geq 1$ on $\Gamma$.

    \textit{Universal approximation:}
    Using the triangle inequality leads to, for any $x \in \mathcal{D}$:
    \begin{multline*}
        | V(x) - \hat{V}(x) | \leq | V(x) - \tilde{V}(x) | + | \tilde{V}(x) - \tilde{V}_\alpha(x) | \\
        \hfill+ |\tilde{V}_\alpha(x) - \hat{V}(x)| \\
        \leq \beta M_q + (\alpha-1) |\tilde{V}(x)| + \eta
    \end{multline*}
    where $M_q = \max_{x \in \mathcal{D}} \tfrac{1}{2} x^{\top} P x$. Since $|\tilde{V}| \leq M_V$ on $\mathcal{D}$ and leveraging \eqref{eq:alpha_rho}, we get $| V(x) - \hat{V}(x) | \leq \beta M_q + (\alpha - 1) M_V + \eta \leq \beta M_q + \frac{\varepsilon}{3} + \eta$. Selecting $\eta < \min\left(\tfrac{1}{3}\varepsilon, \alpha \rho - 1, \eta_0, \tfrac{\alpha \ell_\rho}{2 M_f}, \alpha-1 \right)$ and $\beta < \frac{\varepsilon}{3 M_q}$ ensure that \eqref{eq:universal_approximation_in_V_D} holds.
\end{proof}

\begin{corollary}
    \label{cor:V_N}
    Under Assumption~\ref{ass:linearization} with $f \in C^2$ around the origin, a Taylor Lyapunov network can pointwise approximate a continuous positive definite solution to 
    \begin{equation}
        \label{eq:zubov_D}
        \forall x \in \mathcal{D}^\circ, \quad \nabla V_m(x) \cdot f(x) = - \omega(x) (1 - V_m(x))
    \end{equation}
    arbitrarily well with $\left.V_m\right|_\Gamma \geq 1$ for $\omega(x) = O(\|x\|^2)$ continuous positive definite on the closure of $\mathcal{R}_\mathcal{D}^*$. 
\end{corollary}
\begin{proof}
    From Theorem~\ref{theo:regularity_W}, $W$ as defined in \eqref{eq:W} is a continuous solution to the aforementioned PDE while being $C^4$ around the origin, and $C^1$ on any compact set included in $\mathcal{R}_D^*$. Moreover, any solution to the PDE with the proposed boundary conditions is a maximal Lyapunov function on $\mathcal{D}$ since all solutions are uniquely defined on $\mathcal{R}_\mathcal{D}^*$ and larger than or equal to $1$ outside. 
    
    Therefore, we can apply Theorem~\ref{theo:universal_approximation_VN} and $V_N$ approximates a maximal Lyapunov function arbitrarily well. 
\end{proof}

\section{Practical Implementation}
\label{sec:practical_implementation}

We cannot directly leverage $W$ as in \eqref{eq:W} because $d_\Gamma$ is unknown and it requires computing the trajectory $\pi(x, \cdot)$ for many $x \in \mathcal{D}$.

Instead, we will leverage Corollary~\ref{cor:V_N} and use the PDE characterization. As initially proposed in \citet{liu2023physics}, we will learn a solution $\hat{V}_{P,\Theta}$ in the form of a Taylor neural Lyapunov function to the Zubov equation \eqref{eq:zubov_D} with the added constraint that $\hat{\omega}_\Phi$ should be positive definite, quadratic around the origin, and such that $\hat{V}_{P,\Theta} \geq 1$ on $\Gamma = \{ x \in \partial \mathcal{D} \mid f(x) \cdot \vec{n}(x) \geq 0 \}$. The positive definiteness of $\hat{V}_{P,\Theta}$ is not required since it is locally positive definite around the origin by design and, by \citet[Lemma~3]{vannelli1985maximal}, it is automatically positive definite on $\mathcal{R}_\mathcal{D}(\hat{V}_{P,\Theta})$ provided it is a solution to \eqref{eq:zubov_D}. For the neural network $\hat{\omega}_\Phi$, we will consider a Lyapunov-net \citep{gaby2022lyapunovnet} which also ensures dominance by a quadratic function near the origin.

\begin{remark}
    Note that, in practice, we select the surrogate for the maximal Lyapunov function as $\hat{V}_{P,\Theta}(x) = \min(1, \texttt{TNLF}_{P,\Theta}(x))$, where $\texttt{TNLF}_{P,\Theta}$ stands for a Taylor Lyapunov network parameterized by $\Theta$ and $P$ so that it automatically satisfies \eqref{eq:zubov_D} almost everywhere outside of $\mathcal{R}(\hat{V}_{P,\Theta})$ while keeping the approximation properties inside.
\end{remark}

This leads to the following optimization problem:
\begin{equation}
    \label{eq:optimization_under_constraints}
    P_*, \Theta_* \in \begin{array}[t]{cl}
        \arg\min_{P,\Theta} \min_\Phi & \displaystyle \mathcal{C}_1(P, \Theta, \Phi) \\
        \text{subject to} & \mathcal{C}_2(P, \Theta) = 0,
    \end{array}
\end{equation}
with 
\[
    \mathcal{C}_1(P, \Theta) = \int_{\mathcal{R}_D(\hat{V}_{P,\Theta})} \hspace*{-0.6cm} \left( L_f \hat{V}_{P,\Theta}(x) + \hat{\omega}_\Phi(x) (1 - \hat{V}_{P,\Theta}(x)) \right)^2 dx,
\]
\[
    \mathcal{C}_2(P, \Theta) = \int_{\Gamma} \left[1 - \hat{V}_{P,\Theta}(s)\right]_+^2 ds.
\]
where $L_fV = \nabla V \cdot f$ and $[x]_+ = \max(0, x)$. Following the penalty method \citep{bertsekas1976penalty}, the loss to minimize is therefore $\mathcal{L}_{\lambda_1,\lambda_2}(P, \Theta) = \lambda_1 \mathcal{C}_1(P, \Theta, \Phi) + \lambda_2 \mathcal{C}_2(\Theta, \Phi)$, with a maximization on $\lambda_1, \lambda_2$.

It is not possible to use a gradient-descent method to reach a local minimum for this problem since computing the gradient is numerically intractable. Instead, we will sample the loss to obtain the empirical loss by estimating the integrals using a Monte Carlo approach. This is straightforward for the two constraints, since we just need to sample points in $\mathcal{D}$ for $\mathcal{C}_1$ and on $\partial \mathcal{D}$ for $\mathcal{C}_2$, and restrict them to their respective domains.

We will leverage the primal-dual method described in \citet{barreau2025accuracy} and the adversarial training presented in \citet{jeong2026mitigating} to enforce the constraints by resampling and increasing the weights $\lambda_1$ and $\lambda_2$ during training. We will also periodically project $P$ to its admissible cone to ensure that $A^\top P + PA \prec 0$, i.e., $P \leftarrow \arg\min_{\hat{P} \succ 0} \| P - \hat{P} \|^2$ such that $A^{\top}\hat{P} + \hat{P} A \prec 0$. This is a semidefinite programming optimization problem (after applying the Schur complement), which can be solved efficiently. The full details for the implementation are given on the GitHub page \href{https://github.com/mBarreau/TaylorLyapunov}{TaylorLyapunov}.

\begin{remark}
    The proposed training is fully unsupervised: we do not use data points obtained via approximating \eqref{eq:W}. A supervised approach, as in \citet{liu2023physics}, is possible by replacing $d_\Gamma$ with a $C^2$ approximation of the distance to $\Gamma$, simulating trajectories numerically, and truncating the integral to a time $T > 0$; the resulting ``noisy" estimates of $W$ can guide the training.
\end{remark}

\section{Numerical Simulations and Discussion}
\label{sec:numerical}

In this section, we present and discuss the results of applying our proposed method to systems that are both globally and locally stable and have different dimensions. In addition, we compare our solution with state-of-the-art alternatives. For the computation, we chose a large value for $p$ such that $\mathcal{D}$ becomes very close to the square $[-1, 1]^n$. 
Note that the domain considered in the examples is not $\mathcal{D}$; however, the state variables were scaled for computation and then scaled back for plotting. The hyperparameters are provided in the GitHub repository.

Since we did not use a validation software as in LyzNet \citep{liu2023physics,tool_lyznet}, we instead sampled many points in $\mathcal{D}$ and rescaled the Lyapunov function such that $V_N$ is a Lyapunov function on $\mathcal{D}$ at these points. This rescaling enables fair comparisons with other methods. 

We evaluated the maximal RoA coverage by sampling points $\mathcal{D}_s = \{ x_i \}_i$ uniformly in $\mathcal{D}$, and, using the maximal Lyapunov function $W$ defined in \eqref{eq:W}, we fixed a threshold $M$ large and constructed a sampling of the RoA as $\tilde{\mathcal{R}}_\mathcal{D}^* = \{ x \in \mathcal{D} \ | \ W(x) < M \}$. A coverage of the maximal RoA is the percentage of points in $\tilde{\mathcal{R}}_\mathcal{D}^*$ such that $V_N < 1$.

\subsection{Globally stable system with a non-quadratic Lyapunov function}

Consider the following system:
\begin{equation} \label{eq:globally_stable2}
    \left\{
        \begin{array}{l}
             \dot{x}_1(t) = - x_1(t) + x_1(t) x_2(t), \\
             \dot{x}_2(t) = - x_2(t).
        \end{array}
    \right.
\end{equation}

The maximal Lyapunov function for this system is not quadratic and is therefore more challenging to approximate \citep{ahmadi2011globally}. Note as well that, no matter the size of the square $\mathcal{D}$, the domain will not be forward invariant, which also restricts the estimate of the RoA. 

\begin{figure*}
    \centering
    \subfloat[Lyapunov function]{\includegraphics[width=0.45\textwidth]{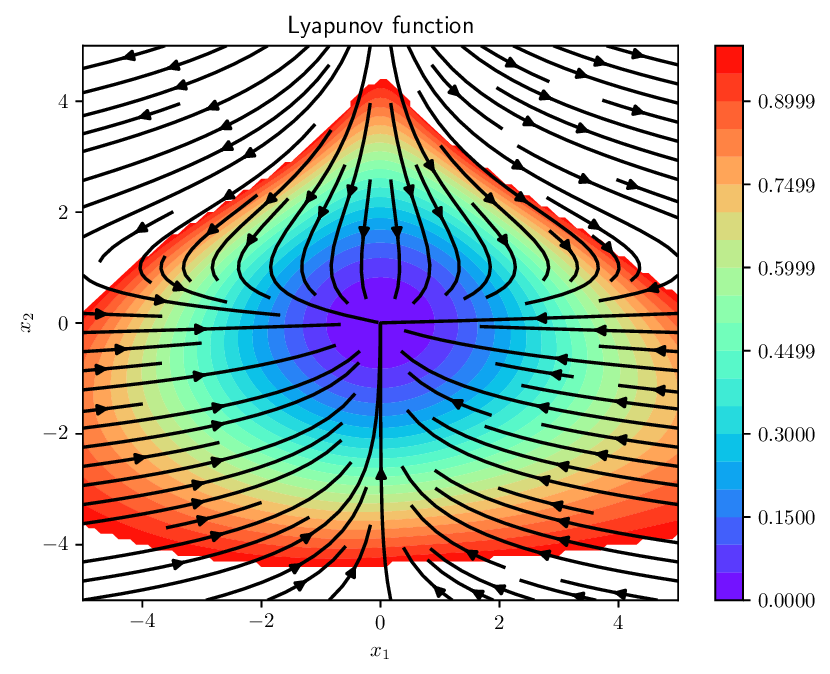}}
    \hfil
    \subfloat[Orbital derivative $\dot{V}$]{\includegraphics[width=0.45\textwidth]{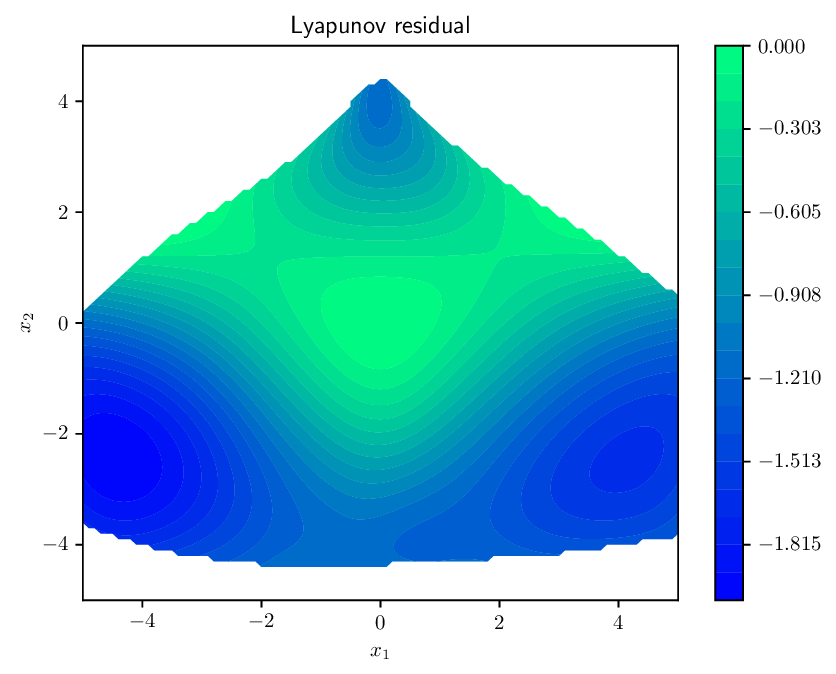}}
    \caption{Taylor Lyapunov network function for globally stable system~\eqref{eq:globally_stable2}. The color scale represents the level sets of the Lyapunov function. The estimated maximal RoA is colored. Arrows indicate the flow of the original system.}
    \label{fig:globally_stable2}
\end{figure*}

The method presented in this paper successfully learns a basin of attraction that covers $83.8\%$ of the RoA in $\mathcal{D}$. These high numbers indicate that the neural network residuals are being learned correctly, enabling the construction of a high-order Lyapunov function. In Fig.~\ref{fig:globally_stable2}, we can also observe that the procedure has identified a forward-invariant set, as the top corners are not part of the basin of attraction, as expected.

\subsection{Locally stable equilibrium point} 

The second system considered is the model of a generator \citep[Example~11.2]{glad2018control}, described as follows:
\begin{equation} \label{eq:locally_stable}
    \left\{
        \begin{array}{l}
             \dot{x}_1(t) = x_2(t), \\
             \dot{x}_2(t) = -\sin(x_1(t)) - 5 x_2(t).
        \end{array}
    \right.
\end{equation}

Since there are multiple equilibrium points, this origin is not globally stable but locally stable because it satisfies Assumption~\ref{ass:linearization}. In \citet[Example~12.6]{glad2018control}, they provide the Lyapunov function $V_{loc}(x_1, x_2) = 0.5 x_2^2 + 1 - \cos(x_1)$ which gives a rather conservative basin of attraction. 

The result is displayed in Fig.~\ref{fig:locally_stable}. Once again, the coverage is good for this example ($81.3\%$) and much larger than the one obtained using $V_{loc}$. We can see that the level lines are not elliptical, which implies that the Lyapunov function is not quadratic. This indicates that higher-order terms in the Taylor decomposition have been learned successfully. However, the inflow into the basin of attraction is highly localized around $x_2 = 0$, making it challenging for the optimizer to estimate the RoA.

\begin{figure}
    \centering
    \includegraphics[width=0.45\textwidth]{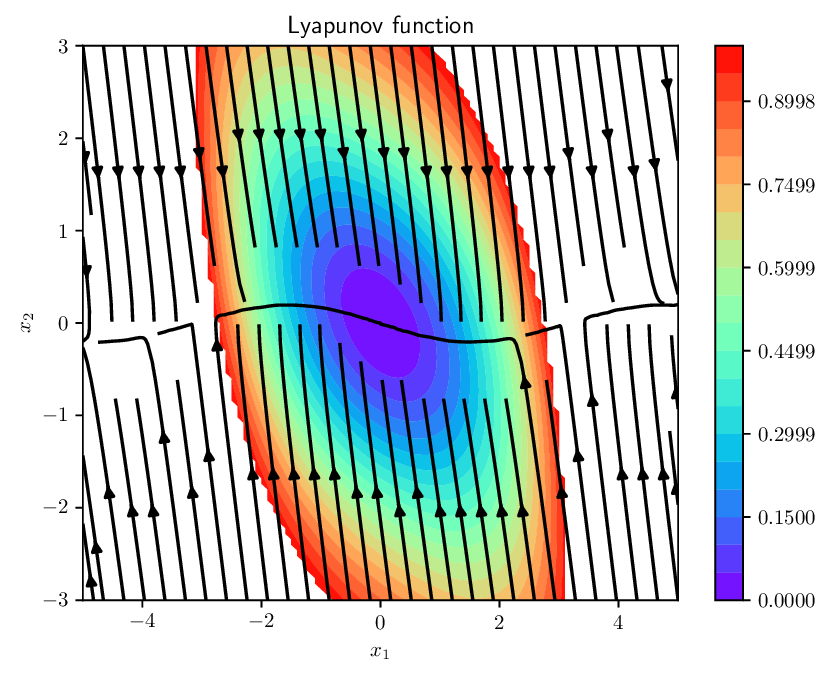}
    \caption{Taylor Lyapunov network function for locally stable system~\eqref{eq:locally_stable}. The color scale represents the sub-level sets of the Lyapunov function. The estimated RoA is the colored area. Arrows indicate the flow of the original system.}
    \label{fig:locally_stable}
\end{figure}

\subsection{Van der Pol oscillator} The last 2D-example considers a special case of Van der Pol oscillator \citep{van1920theory}:
\begin{equation} \label{eq:vanderpol}
    \left\{
        \begin{array}{l}
             \dot{x}_1(t) = -x_2(t), \\
             \dot{x}_2(t) = x_1(t) - \mu(1 - x_1(t)^2) x_2(t).
        \end{array}
    \right.
\end{equation}
We investigate the case $\mu = 1$. This system has a polynomial structure, which makes the use of SOS Lyapunov functions possible \citep{henrion2013convex}. The RoA has a nonconvex shape, which becomes stiffer as $\mu$ increases. 

The result is shown in Fig.~\ref{fig:vanderpol} with a coverage of $98.6\%$. As in the previous example, the maximum RoA is well estimated. The obtained result surpasses the SOS result (dashed line in the figure). We also tried the case $\mu = 1.4$ with a coverage of $81.7\%$. The method presented here is sensitive to the stiffness of the original system, since more points should be sampled in regions with large variations, and, due to the spectral bias, this explains the decrease with $\mu$. The estimation collapses for $\mu = 3$ to only $0.1\%$.

\begin{figure}
    \centering
    \includegraphics[width=0.45\textwidth]{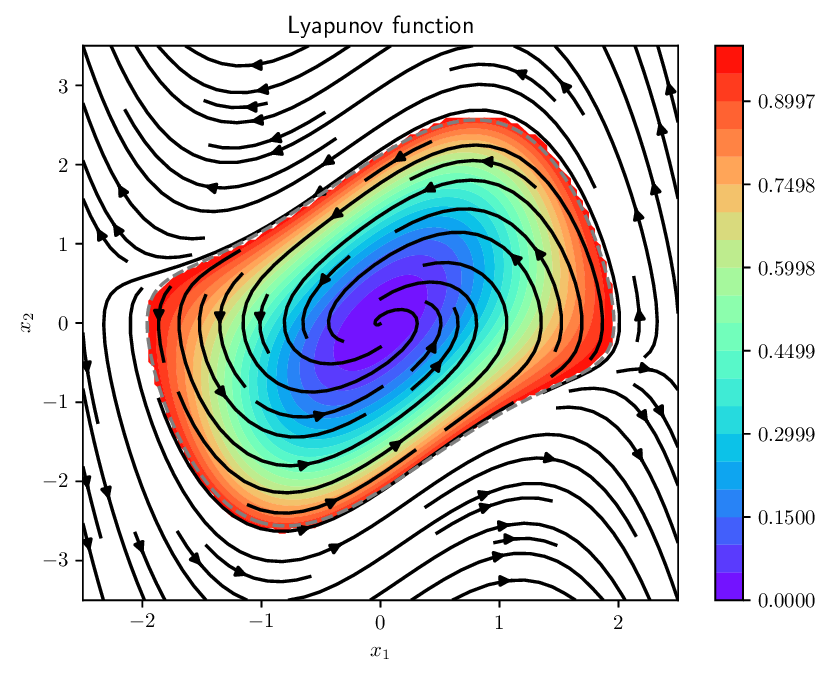}
    \caption{Taylor Lyapunov network function for the Van der Pol oscillator with $\mu = 1$. The color scale represents the sub-level sets of the Lyapunov function. The estimated RoA is the colored area. Arrows indicate the flow of the original system. The dashed line is the estimated RoA obtained using SOS.}
    \label{fig:vanderpol}
\end{figure}

\subsection{High-dimensional example}
To conclude the simulation section, we consider the following $10$-dimensional system from \citet{LarsGrune} to showcase the efficiency in high dimensions:
\begin{equation} \label{eq:10d_system}
    \left\{
        \begin{array}{l}
             \dot{x}_{2k-1} = -x_{2k-1} + 0.5\,x_{2k} + g_k(x), \\
             \dot{x}_{2k} = -0.5\,x_{2k-1} - x_{2k} + h_k(x),
        \end{array}
    \right. \quad k \in \{1, \dots, 5\},
\end{equation}
with $g_1(x) = -0.1\,x_9^2$, $g_2(x) = -0.1\,x_1^2$, $g_3(x) = 0.1\,x_7^2$, $h_5(x) = 0.1\,x_2^2$, and $g_4 = g_5 = h_1 = \cdots = h_4 = 0$.
This system is globally stable and admits a quadratic function. However, inside a fixed square domain $\mathcal{D}$, estimating the maximal RoA is challenging, as it is unlikely to be ellipsoidal (indeed, ignoring forward invariance, the whole $\mathcal{D}$ should be selected).
LyzNet \citep{liu2023physics} applied to this system computed a quadratic Lyapunov function, which, as expected, is suboptimal. On the other hand, our proposed method learned a RoA with $99.9\%$ coverage. The training is, however, rather slow due to the high number of sampling points required in a high-dimensional setting.

\subsection{Ablation study and comparison}

We evaluate the impact of adopting the Taylor Lyapunov network function architecture versus the traditional one used in \citet{gaby2022lyapunovnet} in Table~\ref{tab:comparison}. None of these methods use data, unlike the supervised results reported in \citet{liu2023physics} and \citet{tool_lyznet}, which are more stable and often yield higher coverage \citep{liu2023physics}. This table is therefore an ablation study showcasing the effect of the tailored neural architecture and of the adversarial sampling strategy in the unsupervised case.

The choice of neural network architecture has significant implications. Indeed, as noted in \citet{gaby2022lyapunovnet}, satisfying the negative definiteness of the orbital derivatives at the origin requires fine sampling for traditional neural networks, which often fail around the origin even if the Lyapunov function is well estimated elsewhere. The rescaling done after training to ensure that the learned function is a Lyapunov function creates a steep Lyapunov function, explaining the very low coverage for standard approaches.

\begin{table}
    \centering
    \caption{Ablation study and comparison with other algorithms (Lyapunov-Net: \citet{gaby2022lyapunovnet}, TNLN: Taylor Neural Lyapunov Network, LyZNet: training of \citet{tool_lyznet}, Adv.: adversarial training). Entries are the RoA coverage for each system.}
    \setlength{\tabcolsep}{3pt}
    \begin{tabular}{l|l|ccc}
        Architecture & Training & \eqref{eq:globally_stable2} & \eqref{eq:locally_stable} & \eqref{eq:vanderpol}, $\mu = 1$ \\
        \hline
        \multirow{2}{*}{Lyapunov-Net} & LyZNet & $0.1\%$ & $\simeq 0\%$ & $\simeq 0\%$ \\
        & Adv. & $0.2\%$ & $0.11\%$ & $0.04\%$ \\
        \hline
        \multirow{2}{*}{TNLN} & LyZNet & $76.1\%$ & $79.1\%$ & $\simeq 0\%$ \\
        & Adv. & $83.8\%$ & $81.3\%$ & $98.6\%$ \\
    \end{tabular}
    \label{tab:comparison}
\end{table}

\section{Conclusion and Perspectives}
\label{sec:conclusions}

In this paper, we considered the problem of discovering maximal Lyapunov functions to determine the maximal RoA of a dynamical system.
To this end, we designed a novel neural network architecture that we proved to be a universal approximator for (maximal) Lyapunov functions.
We compared the advantages of the tailored neural architecture with the state of the art by evaluating the accuracy of their RoA approximations across different systems with global or local stability and varying dimensions. The proposed approach proved to be an effective tool for stability analysis of non-linear dynamical systems in the unsupervised setting.

Several future research directions can be explored, building on the work presented here.
Firstly, the architecture of the Taylor Lyapunov network function can be modified to verify, by design, additional properties of Lyapunov functions, thereby reducing the complexity of the learning problem. Changing the architecture can also enhance computational efficiency and numerical stability, especially for stiff dynamical systems or those involving high dimensions.
Secondly, the verification approach proposed in \citet{liu2023physics} can be tailored to the proposed architecture to formally guarantee that the trained function is indeed a Lyapunov function.
Thirdly, our approach can simplify the controller synthesis method developed in \citet{gaby2022lyapunovnet}, \citet{wang2024actor}, and \citet{liu2025formally} by providing an accurate unsupervised learning approach.

\begin{ack}                               
During the preparation of this manuscript, the authors used Claude Sonnet 5
(Anthropic, July-October 2026) to improve the clarity of selected paragraphs, to identify weaknesses, and to shorten the final manuscript to 16 pages. The authors reviewed, edited, and approved all text and take full responsibility for the content.
\end{ack}

\bibliographystyle{apalike}
\bibliography{refs}

@inproceedings{325708,
  author    = {Long, Y. and Bayoumi, M. M.},
  title     = {Feedback stabilization: Control {L}yapunov functions modelled by neural networks},
  booktitle = {IEEE Conference on Decision and Control (CDC)},
  volume    = {3},
  year      = {1993}
}

@article{abate2020formal,
  author  = {Abate, A. and Ahmed, D. and Giacobbe, M. and Peruffo, A.},
  title   = {Formal synthesis of {L}yapunov neural networks},
  journal = {IEEE Control Systems Letters},
  year    = {2020}
}

@inproceedings{ahmadi2011globally,
  author    = {Ahmadi, A. A. and Krstic, M. and Parrilo, P. A.},
  title     = {A globally asymptotically stable polynomial vector field with no polynomial {L}yapunov function},
  booktitle = {IEEE Conference on Decision and Control and European Control Conference (CDC-ECC)},
  pages     = {7579--7580},
  year      = {2011}
}

@article{ahmadi2019dsos,
  author  = {Ahmadi, A. A. and Majumdar, A.},
  title   = {{DSOS} and {SDSOS} optimization: More tractable alternatives to sum of squares and semidefinite optimization},
  journal = {SIAM Journal on Applied Algebra and Geometry},
  volume  = {3},
  number  = {2},
  pages   = {193--230},
  year    = {2019}
}

@inproceedings{ahmed2020automated,
  author    = {Ahmed, D. and Peruffo, A. and Abate, A.},
  title     = {Automated and sound synthesis of {L}yapunov functions with {SMT} solvers},
  booktitle = {International Conference on Tools and Algorithms for the Construction and Analysis of Systems (TACAS)},
  pages     = {97--114},
  year      = {2020}
}

@article{akhtar2026learning,
  author  = {Akhtar, A. and Barreau, M.},
  title   = {Learning neural maximal {L}yapunov functions on {SO}(n)},
  journal = {IEEE Control Systems Letters},
  year    = {2026}
}

@article{angeli1999forward,
  author  = {Angeli, D. and Sontag, E. D.},
  title   = {Forward completeness, unboundedness observability, and their {L}yapunov characterizations},
  journal = {Systems \& Control Letters},
  volume  = {38},
  number  = {4--5},
  pages   = {209--217},
  year    = {1999}
}

@misc{barreau2025accuracy,
  author       = {Barreau, M. and Shen, H.},
  title        = {A control perspective on training {PINNs}},
  howpublished = {arXiv preprint arXiv:2501.18582},
  year         = {2025}
}

@article{bertsekas1976penalty,
  author  = {Bertsekas, D. P.},
  title   = {On penalty and multiplier methods for constrained minimization},
  journal = {SIAM Journal on Control and Optimization},
  volume  = {14},
  number  = {2},
  pages   = {216--235},
  year    = {1976}
}

@inproceedings{cheng2024learning,
  author    = {Cheng, J. and Wang, R. and Manchester, I. R.},
  title     = {Learning stable and passive neural differential equations},
  booktitle = {IEEE Conference on Decision and Control (CDC)},
  pages     = {7859--7864},
  year      = {2024}
}

@book{coleman2012calculus,
  author    = {Coleman, R.},
  title     = {Calculus on Normed Vector Spaces},
  publisher = {Springer},
  year      = {2012}
}

@article{dawson2023safe,
  author  = {Dawson, C. and Gao, S. and Fan, C.},
  title   = {Safe control with learned certificates: A survey of neural {L}yapunov, barrier, and contraction methods for robotics and control},
  journal = {IEEE Transactions on Robotics},
  year    = {2023}
}

@inproceedings{gaby2022lyapunovnet,
  author    = {Gaby, N. and Zhang, F. and Ye, X.},
  title     = {{L}yapunov-{N}et: A deep neural network architecture for {L}yapunov function approximation},
  booktitle = {IEEE Conference on Decision and Control (CDC)},
  year      = {2022}
}

@book{glad2018control,
  author    = {Glad, T. and Ljung, L.},
  title     = {Control Theory},
  publisher = {CRC Press},
  year      = {2000}
}

@article{goertzen2026hardnet,
  author  = {Goertzen, A. and Alim, K. and Min, Y. and Azizan, N.},
  title   = {{HardNet++}: Nonlinear constraint enforcement in neural networks},
  journal = {IEEE Control Systems Letters},
  year    = {2026}
}

@misc{harting2026how,
  author       = {Harting, A. and Johansson, K. H. and Tarbouriech, S. and Barreau, M.},
  title        = {Learning locally stable neural {ODEs} and their region of attraction},
  howpublished = {arXiv preprint},
  year         = {2026}
}

@article{henrion2013convex,
  author  = {Henrion, D. and Korda, M.},
  title   = {Convex computation of the region of attraction of polynomial control systems},
  journal = {IEEE Transactions on Automatic Control},
  volume  = {59},
  number  = {2},
  pages   = {297--312},
  year    = {2013}
}

@article{hornik1991approximation,
  author  = {Hornik, K.},
  title   = {Approximation capabilities of multilayer feedforward networks},
  journal = {Neural Networks},
  volume  = {4},
  number  = {2},
  pages   = {251--257},
  year    = {1991}
}

@article{jeong2026mitigating,
  author  = {Jeong, M. and Seo, G. and Hwang, E.},
  title   = {Mitigating instability in high residual adaptive sampling for {PINNs} via {L}angevin dynamics},
  journal = {Advances in Neural Information Processing Systems},
  volume  = {38},
  pages   = {121574--121620},
  year    = {2026}
}

@article{jongeneel2024continuation,
  author  = {Jongeneel, W. and Schwan, R.},
  title   = {On continuation and convex {L}yapunov functions},
  journal = {IEEE Transactions on Automatic Control},
  volume  = {69},
  number  = {10},
  pages   = {6895--6906},
  year    = {2024}
}

@article{karniadakis2021physics,
  author  = {Karniadakis, G. E. and Kevrekidis, I. G. and Lu, L. and Perdikaris, P. and Wang, S. and Yang, L.},
  title   = {Physics-informed machine learning},
  journal = {Nature Reviews Physics},
  year    = {2021}
}

@book{khalil2002nonlinear,
  author    = {Khalil, H. K.},
  title     = {Nonlinear Systems},
  edition   = {3rd},
  publisher = {Prentice Hall},
  year      = {2002}
}

@article{kolter2019learning,
  author  = {Kolter, J. Z. and Manek, G.},
  title   = {Learning stable deep dynamics models},
  journal = {Advances in Neural Information Processing Systems},
  volume  = {32},
  year    = {2019}
}

@article{LarsGrune,
  author  = {Gr{\"u}ne, L.},
  title   = {Computing {L}yapunov functions using deep neural networks},
  journal = {Journal of Computational Dynamics},
  volume  = {8},
  number  = {2},
  pages   = {131--152},
  year    = {2021}
}

@book{lee_introduction_2012,
  author    = {Lee, J. M.},
  title     = {Introduction to Smooth Manifolds},
  publisher = {Springer},
  address   = {New York},
  year      = {2012}
}

@inproceedings{lima2024physics,
  author    = {Lima, D. R. and Ushirobira, R. and Efimov, D.},
  title     = {A physics-informed neural network method to approximate homogeneous {L}yapunov functions},
  booktitle = {European Control Conference (ECC)},
  pages     = {174--179},
  year      = {2024}
}

@article{liu2023physics,
  author  = {Liu, J. and Meng, Y. and Fitzsimmons, M. and Zhou, R.},
  title   = {Physics-informed neural network {L}yapunov functions: {PDE} characterization, learning, and verification},
  journal = {Automatica},
  volume  = {175},
  pages   = {112193},
  year    = {2025}
}

@inproceedings{liu2025formally,
  author    = {Liu, J. and Fitzsimmons, M. and Zhou, R. and Meng, Y.},
  title     = {Formally verified physics-informed neural control {L}yapunov functions},
  booktitle = {American Control Conference (ACC)},
  pages     = {1347--1354},
  year      = {2025}
}

@misc{meng2026characterization,
  author       = {Meng, Y. and Liu, J.},
  title        = {Characterization of safe stabilization and control {L}yapunov-barrier functions via {Z}ubov equation formulation},
  howpublished = {arXiv preprint arXiv:2604.00941},
  year         = {2026}
}

@inproceedings{meng2026towards,
  author    = {Meng, Y. and Liu, J.},
  title     = {Towards learning and verifying maximal {L}yapunov-barrier functions with a {Z}ubov {PDE} formulation},
  booktitle = {European Control Conference (ECC)},
  pages     = {1637--1644},
  year      = {2026}
}

@misc{min2024hardnet,
  author       = {Min, Y. and Azizan, N.},
  title        = {{HardNet}: Hard-constrained neural networks with universal approximation guarantees},
  howpublished = {arXiv preprint arXiv:2410.10807},
  year         = {2024}
}

@article{mishra2023estimates,
  author  = {Mishra, S. and Molinaro, R.},
  title   = {Estimates on the generalization error of physics-informed neural networks for approximating {PDEs}},
  journal = {IMA Journal of Numerical Analysis},
  volume  = {43},
  number  = {1},
  pages   = {1--43},
  year    = {2023}
}

@article{nagumo1942lage,
  author  = {Nagumo, M.},
  title   = {{\"U}ber die {L}age der {I}ntegralkurven gew{\"o}hnlicher {D}ifferentialgleichungen},
  journal = {Proceedings of the Physico-Mathematical Society of Japan, 3rd Series},
  volume  = {24},
  pages   = {551--559},
  year    = {1942}
}

@inproceedings{pimlcontrol,
  author    = {Nghiem, T. X. and Drgo{\v{n}}a, J. and Jones, C. and others},
  title     = {Physics-informed machine learning for modeling and control of dynamical systems},
  booktitle = {American Control Conference (ACC)},
  year      = {2023}
}

@inproceedings{richards2018lyapunov,
  author    = {Richards, S. M. and Berkenkamp, F. and Krause, A.},
  title     = {The {L}yapunov neural network: Adaptive stability certification for safe learning of dynamical systems},
  booktitle = {Conference on Robot Learning (CoRL)},
  year      = {2018}
}

@incollection{Sideris2013,
  author    = {Sideris, T. C.},
  title     = {Dependence on initial conditions and parameters},
  booktitle = {Ordinary Differential Equations and Dynamical Systems},
  series    = {Atlantis Studies in Differential Equations},
  volume    = {2},
  pages     = {89--94},
  publisher = {Atlantis Press},
  year      = {2013}
}

@article{tee2009barrier,
  author  = {Tee, K. P. and Ge, S. S. and Tay, E. H.},
  title   = {Barrier {L}yapunov functions for the control of output-constrained nonlinear systems},
  journal = {Automatica},
  volume  = {45},
  number  = {4},
  pages   = {918--927},
  year    = {2009}
}

@inproceedings{tool_lyznet,
  author    = {Liu, J. and Meng, Y. and Fitzsimmons, M. and Zhou, R.},
  title     = {{LyZNet}: A lightweight {P}ython tool for learning and verifying neural {L}yapunov functions and regions of attraction},
  booktitle = {ACM International Conference on Hybrid Systems: Computation and Control (HSCC)},
  year      = {2024}
}

@article{van1920theory,
  author  = {{van der Pol}, B.},
  title   = {A theory of the amplitude of free and forced triode vibrations},
  journal = {Radio Review},
  volume  = {1},
  pages   = {701--710},
  year    = {1920}
}

@article{vannelli1985maximal,
  author  = {Vannelli, A. and Vidyasagar, M.},
  title   = {Maximal {L}yapunov functions and domains of attraction for autonomous nonlinear systems},
  journal = {Automatica},
  volume  = {21},
  number  = {1},
  pages   = {69--80},
  year    = {1985}
}

@article{wang2024actor,
  author  = {Wang, J. and Fazlyab, M.},
  title   = {Actor-critic physics-informed neural {L}yapunov control},
  journal = {IEEE Control Systems Letters},
  year    = {2024}
}

@inproceedings{wang2024monotone,
  author    = {Wang, R. and Dvijotham, K. and Manchester, I. R.},
  title     = {Monotone, bi-{L}ipschitz, and {P}olyak-{{\L}}ojasiewicz networks},
  booktitle = {International Conference on Machine Learning (ICML)},
  pages     = {50379--50399},
  year      = {2024}
}

@inproceedings{zenkin2025neural,
  author    = {Zenkin, A. and Ushirobira, R. and Efimov, D. and Bobtsov, A.},
  title     = {A neural network-based construction of {L}yapunov functions based on homogeneous approximations},
  booktitle = {IEEE Conference on Decision and Control (CDC)},
  pages     = {5582--5587},
  year      = {2025}
}

@article{zhou2022neural,
  author  = {Zhou, R. and Quartz, T. and De Sterck, H. and Liu, J.},
  title   = {Neural {L}yapunov control of unknown nonlinear systems with stability guarantees},
  journal = {Advances in Neural Information Processing Systems},
  volume  = {35},
  year    = {2022}
}

@article{zhou2024physics,
  author  = {Zhou, R. and Fitzsimmons, M. and Meng, Y. and Liu, J.},
  title   = {Physics-informed extreme learning machine {L}yapunov functions},
  journal = {IEEE Control Systems Letters},
  volume  = {8},
  pages   = {1763--1768},
  year    = {2024}
}

@book{zubov1961methods,
  author    = {Zubov, V. I.},
  title     = {Methods of {A}. {M}. {L}yapunov and Their Application},
  volume    = {4439},
  publisher = {U.S. Atomic Energy Commission},
  year      = {1961}
}

\appendix

\section{Proof of Lemma~\ref{lem:uniform_CV}}
\label{app:uniform_CV}

    If $V$ is quadratic function on $\mathcal{D}$, then $V \in \mathcal{V}_N$ for any $N$. Just pick $P = \nabla^2 V(0)$, $\hat{R}_N \equiv 0$ and $\delta_0$ such that $\mathcal{B}_{\delta_0} \subseteq \mathcal{R}_{\mathcal{D}}(V)$.
    
    In the case $V$ not a quadratic function on $\mathcal{D}$, let $P \triangleq \nabla^2 V(0)$, $c_1 \triangleq \lambda_{\min}(P) > 0$, $c_2 \triangleq -\lambda_{\max}(A^{\top} P + PA) > 0$, and $c_3 \triangleq \max_{x \in \mathcal{D}, i\in\mathcal{I}_n} | R_i(x) | > 0$. Note that $c_3 < \infty$ since $R \in C(\mathcal{D}, \R)$ and $\mathcal{D}$ is compact. 
    For any compact set $K \subset \mathcal{R}_\mathcal{D}(V)$, We also introduce $c_4 \triangleq \max_{x \in K, i \in \mathcal{I}} \| \nabla R_i(x) \| \geq 0$. Note that $c_4 < \infty$ since $\nabla R_i$ are continuous on $\mathcal{R}_\mathcal{D}(V)$.

    \textit{Positive definiteness:}
    Define $\hat{c}_3(\varepsilon
    ) = \max_{x \in \mathcal{D}, i \in \mathcal{I}_n} \| \hat{R}_i \|$ and, following the construction in the proof of Proposition~\ref{prop:universal_approximation_V}, we also get $\hat{c}_3 \xrightarrow[\varepsilon \to 0]{} c_3$, i.e.:
    \begin{equation}
        \label{eq:delta}
        \forall \delta > 0, \exists \varepsilon_1(\delta) > 0, \forall \varepsilon < \varepsilon_1,  \quad c_3 - \varepsilon < \hat{c}_3(\varepsilon) \leq c_3 + \delta.
    \end{equation}

    Using Cauchy-Schwarz, we get $\hat{V}(x) \geq \| x \|^2 ( \tfrac{1}{2} c_1  - m_n \hat{c}_3(\varepsilon) \| x \| )$ for all $x \in \mathcal{D}$. Pick $\delta \in (0, c_3)$ in \eqref{eq:delta}, therefore for all $\varepsilon < \varepsilon_1(\delta)$, we get $0 < \hat{c}_3(\varepsilon) \leq 2c_3$. For $0 < \| x \| \leq \frac{c_1}{2 m_n \hat{c}_3(\varepsilon)}$, we get $\hat{V}(x) > 0$.

    \textit{Definite negativeness of the orbital derivatives}
    Computing the orbital derivative leads to, for $x \in \mathcal{D}$:
    \begin{multline*}
        \nabla \hat{V}(x) \cdot f(x) = x^{\top} P f(x) + \nabla \hat{V}_3(x) \cdot f(x) \\
        = \tfrac{1}{2} x^{\top} \left( A^{\top} P + PA \right) x + x^{\top} P g(x) + \nabla \hat{V}_3(x) \cdot f(x) \\
        \leq -\tfrac{c_2}{2} \|x\|^2 + \| P \| \| x \| \| g(x) \| + \| \nabla \hat{V}_3(x) \| \| f(x) \|.
    \end{multline*}
    Since $f$ is $C^2$ locally around the origin, then $g(x) = O(\| x \|^2)$ in a ball of radius $\bar{\delta} > 0$ around the origin and therefore, there exists $M_g > 0$ such that $\| g(x) \| \leq M_g \| x \|^2$ for $x \in \mathcal{B}_{\bar{\delta}}$, so that $\| P \| \| x \| \| g(x) \| \leq \| P \| M_g \| x \|^3$ on $\mathcal{B}_{\bar{\delta}}$.

    By definition, we get: 
    \[
        \nabla \hat{V}_3(x) = \sum_{i \in \mathcal{I}_n} \nabla \hat{R}_i(x) \prod_{k=1}^n x_k^{i_k} + \sum_{i \in \mathcal{I}_n} \hat{R}_i(x) \nabla \left( \prod_{k=1}^n x_k^{i_k} \right).
    \]
    By Cauchy-Schwartz, we get $|\prod_{k=1}^n x_k^{i_k}| \leq \| x \|^3$ for any $x \in \mathcal{D}$ and $i \in \mathcal{I}_n$. Let $\hat{c}_4(\varepsilon) = \max_{x \in K, i \in \mathcal{I}_n} \| \nabla \hat{R}_i \|$. Since $x \mapsto \prod_{k=1}^n x_k^{i_k}$ is a third degree polynomial for $i \in \mathcal{I}_n$, then $\|\nabla \left( \prod_{k=1}^n x_k^{i_k} \right) \| \leq 3 \| x \|^2$. Therefore, we get:
    \[
        \| \nabla \hat{V}_3(x) \| \leq m_n \hat{c}_4(\varepsilon) \| x \|^3 + 3 m_n \hat{c}_3 \| x \|^2.
    \]
    Combining all the results with $L$-Lipschitzness of $f$, leads to $\nabla \hat{V}(x) \cdot f(x) \leq ( -\tfrac{1}{2} c_2 + K(\varepsilon) \| x \| ) \|x\|^2$, where $K(\varepsilon) = \| P \| M_g + 3 L m_n \hat{c}_3(\varepsilon) + m_n L \hat{c}_4(\varepsilon) c_5$ with $c_5 = \max_{x \in \mathcal{D}} \| x \|$.

    Following the same reasoning as previously, we get that there exists $\varepsilon_2 > 0$ such that for all $\varepsilon \in (0, \varepsilon_2)$, we get $0 \leq \hat{c}_4(\varepsilon) \leq 2 c_4$ (it might be $0$ this time since $c_4 \geq 0$). Consequently, for any $\varepsilon < \min(\varepsilon_1, \varepsilon_2)$, we get $0 < K(\varepsilon) < \bar{K} \triangleq \| P \| M_g + 6 L m_n c_3 + 2 m_n L c_4 c_5$.

    For any $x \in \mathcal{B}_{\bar{\delta}} \setminus \{0\}$ such that $\| x \| \leq \frac{c_2}{2 K(\varepsilon)}$, we get $\nabla \hat{V}(x) \cdot f(x) < 0$.

    \textit{Conclusion:} For $\delta_0 = \min\left( \frac{c_1}{4 m_n c_3}, \frac{c_2}{4 \bar{K}}, \bar{\delta} \right)$ and $\varepsilon_* = \min(\varepsilon_1, \varepsilon_2)$, we get that for all $\varepsilon \in (0, \varepsilon_*)$, implication~\eqref{implication:hatV} holds.

\end{document}